\documentclass[11pt,a4paper]{article}

\usepackage[margin=1.0in]{geometry}

\usepackage{amssymb}
\usepackage{amsmath,amsthm,amsfonts}
\usepackage{float}
\usepackage{algorithm}

\usepackage{subfigure}

\usepackage{graphicx}
\usepackage{diagbox}

 \usepackage{hyperref}

\usepackage{xcolor}

\usepackage{comment}

\usepackage{multicol}

\newtheorem{theorem}{Theorem}
\newtheorem{lemma}{Lemma}[section]

\newtheorem{assumption}{Assumption}
\newtheorem{proposition}{Proposition}
\theoremstyle{definition}
\newtheorem{definition}{Definition}
\theoremstyle{remark}
\newtheorem{remark}{Remark}

\DeclareMathOperator*{\argmin}{arg\,min}

\newcommand{\N}{\mathbb{N}}
\newcommand{\R}[1]{\mathbb{R}^{#1}}
\newcommand{\RR}[2]{\mathbb{R}^{#1\times #2}}

\newcommand{\Lp}{\mathcal{L}_\rho}

\newcommand{\one}{\textbf{1}_{\mathcal{X}}}
\newcommand{\fl}{f_{\lambda}}
\newcommand{\Ll}{L_{\lambda}}

\newcommand{\as}{\alpha_S}
\newcommand{\bs}{\beta_S}
\newcommand{\zero}{\textbf{0}}

\newcommand{\Q}{\mathcal{Q}}
\newcommand{\X}{\mathcal{X}}
\newcommand{\Y}{\mathcal{Y}}

\newcommand{\Pp}{\mathcal{P}^*}
\newcommand{\D}{\mathcal{D}^*}
\newcommand{\F}{\mathcal{F}}
\newcommand{\s}{\mathcal{S}_\tau}

\newcommand{\A}{\mathcal{A}}
\newcommand{\K}{\mathcal{K}}
\newcommand{\I}{\mathcal{I}}
\newcommand{\B}{\mathcal{B}}
\newcommand{\C}{\mathcal{C}}
\newcommand{\T}{\mathcal{T}}

\newcommand{\QQ}{\textbf{Q}}
\newcommand{\XX}{\textbf{X}}
\newcommand{\YY}{\textbf{Y}}
\newcommand{\WW}{\textbf{W}}
\newcommand{\TT}{\textbf{T}}
\newcommand{\MM}{\textbf{M}}
\newcommand{\UU}{\textbf{U}}
\newcommand{\VV}{\textbf{V}}
\newcommand{\CC}{\textbf{C}}
\newcommand{\Ss}{\textbf{S}}

\newcommand{\x}{\textbf{x}}
\newcommand{\y}{\textbf{y}}
\newcommand{\q}{\textbf{q}}
\newcommand{\w}{\textbf{w}}
\newcommand{\z}{\textbf{z}}

\newcommand{\p}{\textbf{p}}
\newcommand{\n}{\textbf{n}}
\newcommand{\uu}{\textbf{u}}
\newcommand{\vv}{\textbf{v}}
\newcommand{\bb}{\textbf{b}}
\newcommand{\cc}{\textbf{c}}

\newcommand{\Rx}{\mathcal{R}_{\mathcal{X}}}
\newcommand{\Ry}{\mathcal{R}_{\Y}}
\newcommand{\Rq}{\mathcal{R}_{\Q}}

\newcommand{\E}{\mathbb{E}}
\newcommand{\V}{\mathbb{V}}

\newcommand{\cv}[2]{\begin{bmatrix}#1\\#2\\\end{bmatrix}}

\title{Faster Projection-Free Augmented Lagrangian Methods via Weak Proximal Oracle}
\author{Dan Garber\footnote{Authors are ordered alphabetically.}\\ {\small \texttt{dangar@technion.ac.il}}
\and
Tsur Livney\footnotemark[\value{footnote}] \\  {\small \texttt{tsur.livney@campus.technion.ac.il}}
\and
Shoham Sabach\footnotemark[\value{footnote}]  \\ {\small \texttt{ssabach@technion.ac.il}} \\ \\
Technion - Israel Institute of Technology
}
\date{}

\begin{document}





\maketitle

\begin{abstract}
    This paper considers a convex composite optimization problem with affine constraints, which includes problems that take the form of minimizing a smooth convex objective function over the intersection of (simple) convex sets, or regularized with multiple (simple) functions. Motivated by high-dimensional applications in which exact projection/proximal computations are not tractable, we propose a \textit{projection-free} augmented Lagrangian-based method, in which primal updates are carried out using a \textit{weak proximal oracle} (WPO). In an earlier work, WPO was shown to be more powerful than the standard \textit{linear minimization oracle} (LMO) that underlies conditional gradient-based methods (aka Frank-Wolfe methods). Moreover, WPO is computationally tractable for many high-dimensional problems of interest, including those motivated by recovery of low-rank matrices and tensors, and optimization over polytopes which admit efficient LMOs. The main result of this paper shows that under a certain curvature assumption (which is weaker than strong convexity), our WPO-based algorithm achieves an ergodic rate of convergence of $O(1/T)$ for both the objective residual and  feasibility gap. This result, to the best of our knowledge, improves upon the $O(1/\sqrt{T})$ rate for existing LMO-based projection-free methods for this class of problems. Empirical experiments on a  low-rank and sparse covariance matrix estimation task and the Max Cut semidefinite relaxation demonstrate that of our method can outperform  state-of-the-art LMO-based Lagrangian-based methods.
\end{abstract}

\section{Introduction}
Throughout the paper, we consider the following minimization problem
\begin{equation}\label{pblm:problem}\tag{OP}
        \min_{\x\in\E_1,\y\in\E_2}f(\x)+\Rx(\x)+\Ry(\y) ~~\textrm{s.t.} ~~ \A\x=\y,
\end{equation}
where $\E_1$ and $\E_2$ are finite Euclidean spaces, $f:\E_1\rightarrow\mathbb{R}$ is a convex and $\beta$-smooth function, $\Rx:\E_1\rightarrow(-\infty,\infty]$ and $\Ry:\E_2\rightarrow(-\infty,\infty]$ are proper, lower semi-continuous and convex functions, and $\A:\E_1\rightarrow\E_2$ is a linear mapping.

Problems that fall into the model \eqref{pblm:problem} appear in many interesting and important active research areas such as machine learning, signal processing, statistics, and more. For example, the recovery of a matrix or a tensor which is both sparse and of low rank is useful in problems such as covariance matrix estimation \cite{andrews1991heteroskedasticity,driscoll1998consistent,richard2012estimation}, graph and image denoising \cite{buades2005non,buades2005review,zhang2017beyond,zhang2020image} and link prediction \cite{liben2003link,lu2011link,zhang2017link}. More applications of Problem \eqref{pblm:problem} can be found in brain mapping \cite{ogawa1992intrinsic,lancaster2000automated,gramfort2013identifying} and multiple sequence alignment \cite{corpet1988multiple,chenna2003multiple,katoh2008recent,yen2016convex}.

An important family of efficient methods for solving problems in the form of model \eqref{pblm:problem} are Lagrangian-based methods, and most notably augmented Lagrangian methods \cite{mizoguchi1960kj,hestenes1969multiplier,powell1969method}. Starting with the classical proximal method of multipliers \cite{rockafellar1976augmented}, and until more recently \cite{sabach2019lagrangian,de2022constrained,dhingra2018proximal, chambolle2016ergodic}, such methods, which are based on proximal/projection computations (due to the nonsmooth functions $\Rx$ and $\Ry$ in Problem \eqref{pblm:problem}) have been successfully developed and corresponding provable convergence rates have been established. However, in many cases of interest such proximal/projection computations are not tractable in high-dimensional problems, for instance when either $\Rx$ or $\Ry$ is  a nuclear norm regularizer for matrices, which underlies many recovery problems of low-rank matrices and tensors (see, for instance, \cite{candes2009exact, candes2011robust, gandy2011tensor}),  or an indicator function for a polytope. Thus, with the growing interest in recent years in so-called \textit{projection-free}  methods, which are mostly based on the use of linear minimization oracles (LMO) instead of proximal/projection oracles through the Frank-Wolfe method (aka conditional gradient, see for instance \cite{jaggi2013revisiting}), and are often much more efficient to implement for high-dimensional problems (e.g., in case $\Rx$ or $\Ry$ is a nuclear norm regularizer or an indicator function for a polytope which captures some well-studied combinatorial structure, see for instance \cite{jaggi2013revisiting, HazanK12}), have been studied \cite{liu2019nonergodic,yurtsever2019conditional,silveti2020generalized}. However, these methods suffer from slow convergence rates compared to their proximal/projection-based counterparts, with the worst-case guaranteed convergence rate being at best $O(1/\sqrt{T})$, where $T$ is the number of iterations executed. This rate is not known to be improvable even under additional standard curvature assumptions such as strong convexity of the function $f(\cdot)$ in Problem \eqref{pblm:problem}\footnote{This is not surprising since it is well-known that in general, and as opposed to projection/proximal-based methods, the Frank-Wolfe method does not benefit from strong convexity, see for instance discussions in \cite{garber2016linearly, garber2016faster, allen2017linear}.}.
A recent attempt to obtain faster projection-free methods under relatively mild assumptions has been made in \cite{gidel2018frank}, however as we discuss in detail in  the appendix (see Section \ref{sec:issue}), there is a major problem with their proof which does not seem easily fixable. We also refer the interested reader to the excellent discussions in  \cite{gidel2018frank} on major issues with other previous attempts to prove faster rates for projection-free methods.

For the simpler problem of minimizing a smooth convex objective function over a convex and compact set, and in particular in case the feasible set is either a polytope or a nuclear norm ball of matrices or a spectrahedron (set of positive semidefinite matrices with unit trace), several recent works showed how simple modifications of the Frank-Wolfe method can lead to provably faster convergence rates, under standard curvature assumptions, see for instance \cite{garber2016linearly, lacoste2015global, beck2017linearly, garber2016faster, allen2017linear, garber2019fast}. Thus, in the context of the significantly more complex Problem \eqref{pblm:problem}, our work considers the following natural question:
\begin{center}
\emph{Can we design a projection-free augmented Lagrangian-based method  that, at least under standard curvature assumptions, improves upon the current $O(1/\sqrt{T})$ convergence rate?}
\end{center}
We answer this question on the affirmative side by providing a projection-free method with a rate of $O(1/T)$, both in terms of the objective function residual and the feasibility gap of the affine constraint in Problem \eqref{pblm:problem}. 

Our approach departs from previous projection-free methods which guarantee only a rate of $O(1/\sqrt{T})$ in two aspects. First, as already suggested, we make a curvature assumption on Problem \eqref{pblm:problem}: we introduce a curvature condition we call \textit{primal quadratic gap} (see definition in the sequel). In particular, this condition holds whenever the smooth function $f(\cdot)$ in Problem \eqref{pblm:problem} is strongly convex, but also holds in case $f(\cdot)$ is a composition of a strongly convex function with a linear transformation (e.g., a least squares objective, which need not be strongly convex) and $\Rx,\Ry$ are indicators for polytopes. Second, while previous projection-free methods rely on the availability of a linear minimization oracle (LMO), in this work we consider a slightly stronger oracle which was already considered in recent works \cite{allen2017linear, garber2018fast, garber2019fast} (these however do not apply to problems such as Problem \eqref{pblm:problem}, which includes affine constraints), namely the \textit{weak proximal oracle} (WPO)\footnote{The term ``weak proximal oracle'' was originally coined in \cite{garber2019fast}.}. In a nutshell, this oracle solves a certain relaxed version of the proximal/projection problem, which can still be much more efficient to solve than the standard proximal/projection problem, but can provide more informative directions than that of the LMO. Two prime examples for the efficiency of implementing the WPO are when (i) $\Rx$ or $\Ry$ is an indicator function for a polytope which admits and efficient LMO, then the WPO could be implemented based on a single call to the LMO of the polytope, and (ii) $\Rx$ or $\Ry$ is an indicator function/regularizer corresponding to the matrix nuclear norm and a unique low-rank optimal solution exists, then implementing the WPO corresponds to a low-rank SVD computation with rank matching that of the low-rank optimal solution, which is much more efficient than proximal/projection computation, which generally requires a full-rank SVD (see a detailed discussion in \cite{allen2017linear, garber2019fast, garber2018fast}).

The combination of the two ingredients: a curvature condition and the weak proximal oracle, to obtain faster convergence rates for projection-free methods should not come as a surprise since it was already instrumental in achieving similar improvements for projection-free methods in settings that do not include affine constraints as in model \eqref{pblm:problem}, see for instance \cite{garber2016linearly, lacoste2015global, garber2016faster, allen2017linear, garber2019fast}\footnote{While technically \cite{garber2016linearly, lacoste2015global} rely on the use of a standard LMO for the feasible set (which they assume to be a polytope), as we show in the sequel, the way they use the output of the LMO to construct the new descent direction is very similar to the implementation of a WPO. In particular, we rely on observations from \cite{garber2016linearly} to construct an efficient WPO for polytopes.}. To the best of our knowledge, this is the first time such an approach is used for a problem of the form of model \eqref{pblm:problem}. 

\begin{table}[H]\renewcommand{\arraystretch}{1.3}
{\footnotesize
    \centering
        \begin{tabular}{ |p{3.8cm}|p{1.2cm}|p{1.2cm}|p{2.3cm}|p{2.1cm}|p{2.5cm}| }
            \hline
            Algorithm & Oracle & Rate & Oracle implementation in polytope setup & Oracle implementation in nuclear norm setup  &  Assumptions\\[3ex] 
            \hline
            Proximal Method of Multipliers \cite{rockafellar1976augmented,sabach2019lagrangian} & proximal & $O(1/T)$ & projection & full-rank SVD & \\[3ex] 
            \hline
            Accelerated Primal Dual  \cite{chambolle2016ergodic} & proximal & $O(1/T^2)$ & projection & full-rank SVD & strong convexity \\[3ex]
            \hline
            Conditional Gradient Augmented Lagrangian \cite{yurtsever2019conditional} & LMO & $O(1/\sqrt{T})$ & LMO & rank-one SVD & \\[3ex] 
            \hline
            This work & weak proximal & $O(1/T)$ & LMO plus convex quadratic opt. over simplex & rank$(\x^*)$-SVD & primal quadratic gap (weaker than strong convexity) \\[3ex] 
            \hline
        \end{tabular}
    \caption{Comparison of augmented Lagrangian-based methods with different optimization oracles for Problem \eqref{pblm:problem}. The ``Rate'' column specifies the convergence rate only in terms of the number of iterations $T$ and suppresses all other quantities. ``Polytope setup'' in the forth column refers to a setting in which $\Rx$ is an indicator function of some compact and convex polytope and $\Ry$ is some proximal friendly function. ``Nuclear norm setup'' in the fifth column refers to a setting in which $\Rx$ is a matrix nuclear norm regularizer and $\Ry$ is some proximal friendly function. $\x^*$ denotes the optimal solution, which is assumed to be unique. Both columns specify the dominating cost of implementing the appropriate optimization oracle for the $\x$ variable.}
    \label{tab:Lagrangian Methods Comparison}}
\end{table}\renewcommand{\arraystretch}{1}

\subsection{Paper organization}
In Section \ref{sec:Preliminaries}, we discuss the augmented Lagrangian approach for solving the Problem \eqref{pblm:problem}, present the Primal Quadratic Gap property (PQG) needed for our algorithm's analysis. We also recall the notion of  the Weak Proximal Oracle that will be used in our algorithm and discuss its implementation in several important scenarios. In the Appendix (see Section \ref{sec:Illustrative Examples}), we give examples of problems of interest, for which our algorithm might be appealing to use. In Section \ref{sec:Algorithm & Analysis}, we develop our algorithm and prove our main rate of convergence result. In Section \ref{sec:Numerical Experiments}, we demonstrate the empirical performance of our algorithm.

\section{Preliminaries}\label{sec:Preliminaries}
\subsection{Notation}
Throughout the paper, we will use the following notation for simplifying the presentation and developments. We use the following compact notations $\q:=\cv{\x}{\y}\equiv(\x,\y)\in\E_1\times\E_2=:\E$, and $\Rq(\q):=\Rx(\x)+\Ry(\y)$. In addition, we define the linear mapping $\K:\E\rightarrow\E_2:=[\A,-\I]$, where $\I$ is an identity linear mapping. This way we can compactly write the constraint $\A\x=\y$ as $\K\q=\textbf{0}$. For any finite Euclidean space $\V$ of dimension $n$, we denote the standard Euclidean inner product of any two points $\x:=[x_1,\dots,x_n]^{\top}$ and $\y:=[y_1,\dots,y_n]^{\top}\in\V$, by $\langle \x,\y\rangle:=\sum_{i=1}^nx_i\cdot y_i$, and we let $\|\x\|:=\sqrt{\langle \x,\x\rangle}$ and $\|\x\|_1:=\sum_{i=1}^n|x_i|$ denote the standard Euclidean norm and the $\ell_1$ norm, respectively. For any two finite Euclidean spaces $\V_1$ and $\V_2$, the \emph{spectral norm} of the linear mapping $\T:\V_1\to\V_2$, is denoted by $\|\T\|:=\max_{\x\in\V_1}\left\{ \|\T\x\| : \, \|\x\| = 1 \right\}$. While both norms are denoted the same, throughout the paper it will be clear from the context which of the norms is used in each appearance. In addition, $\|\XX\|_F$ denotes the Frobenius norm of a matrix $\XX$, $\|\XX\|_{\textrm{nuc}}$ denotes its nuclear norm and $\textrm{tr}(\XX)$ denotes its trace. We denote by $\mathbb{S}^d_+$ the set of all positive semidefinite matrices of size $d\times d$ and by $\s$ the spectrahedron of all matrices in $\mathbb{S}^d_+$ with trace $\tau$. We use the notation $\delta_C(\cdot)$ to denote the indicator function of a set $C$.

\subsection{The Augmented Lagrangian}

The \emph{augmented Lagrangian} (AL) of Problem \eqref{pblm:problem} with a multiplier (dual variable) $\w\in\E_2$ is defined as
\begin{align}
        \Lp(\q,\w) &\equiv\Lp(\x,\y,\w) \nonumber
        :=f(\x)+\Rq(\q)+\langle\w,\K\q\rangle+\frac{\rho}{2}\|\K\q\|^2, \label{eq:Augmented Lagrangian}
\end{align}
 where $\rho>0$ is a penalty parameter associated to the linear equality constraint. Note that the (standard) Lagrangian of Problem \eqref{pblm:problem} is recovered when $\rho = 0$.
\\

We also require the two following standard assumptions which we assume to hold true throughout the paper.

\begin{assumption}\label{asmp:saddle}
    The augmented Lagrangian has a saddle point, i.e., there exists a point $(\q^*,\w^*)\in\E\times\E_2$ satisfying
    \begin{equation}\label{eq:saddle point}
        \Lp(\q^*,\w)\leq\Lp(\q^*,\w^*)\leq\Lp(\q,\w^*), 
    \end{equation}
    for all $\q\in\E$ and $\w\in\E_2$.
\end{assumption}

\begin{assumption}\label{asmp:slater}(Slater's Condition)
    There exist $\x\in ri(dom(f)\cap dom(\Rx))$ and $\y\in ri(dom(\Ry))$ such that $\A\x=\y$, where $ri$ denotes the relative interior of a set.
\end{assumption}
We denote by $\Pp$ the set of optimal solutions of the primal problem \eqref{pblm:problem}, and by $\D$ the set of optimal solutions of the associated dual problem.

Under the above two assumptions, and thanks to the convexity of Problem \eqref{pblm:problem}, strong duality holds. As a result, and thanks to Proposition \ref{prp:saddle points set}, the set of saddle points of $\Lp$ is non-empty and corresponds to the set of \emph{all} pairs $(\q^*,\w^*)$, where $\q^*\in\Pp$ and $\w\in\D$ (see also Appendix \ref{sec:pqg for polytopes}).

In order to find solutions of Problem \eqref{pblm:problem}, we will solve the following equivalent saddle point problem
\begin{equation}\label{pblm:reformulation}
    \min_{\q\in\E}\max_{\w\in\E_2}\Lp(\q,\w),
\end{equation}
whose optimal solutions are the saddle points of $\Lp$.

From now on, we will denote the optimal objective function value by $\Lp(\q^*,\w^*)$, for all saddle points $(\q^*,\w^*)$, and for short we will write $\Lp^*$.

The \emph{smooth} part of the \emph{augmented Lagrangian} (SAL) of Problem \eqref{pblm:problem} is defined, for any $\q\in\E$ and $\w\in\E_2$, by
\begin{equation}\label{eq:smooth part}
    S(\q,\w)\equiv S(\x,\y,\w):=f(\x)+\langle\w,\K\q\rangle+\frac{\rho}{2}\|\K\q\|^2.
\end{equation}
The following result will be essential to our developments in the sequel. Its simple proof is deferred to the appendix.
\begin{lemma}\label{lemma:loose upper bound for K}
     The function $\q \rightarrow S(\q,\w)$, for any fixed $\w \in \E_{2}$, is smooth with parameter $\bs=\beta+\rho(\|\A\|+~1)^2$.
\end{lemma}

\subsubsection{Primal Quadratic Gap}
We now study a new property of the smooth function $\q \rightarrow S(\cdot,\w)$, for any fixed $\w \in \E_{2}$, which we refer to as the \emph{Primal Quadratic Gap} (PQG).
\begin{definition}(Primal Quadratic Gap)\label{def:primal quadratic gap}
    We say that Problem \eqref{pblm:problem} satisfies the \emph{Primal Quadratic Gap} property with a parameter $\as>0$, if for any $\q\in \textrm{dom}(\Rq)$, the point $\q^*:=\argmin_{\q^*\in\Pp}\|\q-\q^*\|^2\in\Pp$ satisfies, for all $\w\in\E_2$, the following inequality 
    \begin{equation}\label{eq:pqg}
    \langle\q^*-\q,\nabla_{\q}S(\q,\w)\rangle\leq S(\q^*,\w)-S(\q,\w)-\frac{\as}{2}\|\q^*-\q\|^2.
    \end{equation}
\end{definition}
This property is a weaker version of the strong convexity property. Here, instead of assuming the inequality \eqref{eq:pqg} holds for any two points, we only require that it holds for any point $\q\in dom(\Rq)$ and the corresponding closest optimal solution of Problem \eqref{pblm:problem}.

\paragraph*{Example 1.} If the function $f(\x)$ is strongly convex, then the function $\q \rightarrow S(\q,\w)$, for a fixed $\w \in\E_{2}$, is strongly convex with a certain parameter $\as$. This implies that in this case the SAL $S(\q,\w)$ satisfies the primal quadratic gap property, with the same parameter $\as$. 
\begin{theorem}\label{thm:first case}
    Suppose that $f : \E_1 \rightarrow \R{}$ is $\alpha$-strongly convex.  Then, Problem \eqref{pblm:problem} admits a unique primal optimal solution $\q^*$ and it satisfies the PQG property with the parameter $\as=\min\{\frac{\alpha}{2},\frac{\alpha\rho}{\alpha+2\rho\|\A\|^2}\}>0$. In particular, the function $\q \rightarrow S(\q,\w)$, for any fixed $\w\in\E_{2}$, is strongly convex.
\end{theorem}
The proof is deferred to the Appendix (see Section \ref{sec:properties proofs}).

\paragraph*{Example 2.} If $\Rq$ is an indicator function for a polytope, we can show that the PQG property holds true even when $f(\cdot)$ need not be strongly convex. Alternatively, we will make the following assumption.
\begin{assumption}\label{asmp:polytope assumption}
   \begin{itemize}
   	\item[(i)] $f \equiv g \circ \B$, where $\B : \E_1 \rightarrow \E_3$ is a linear mapping, and $g : \E_3 \rightarrow \R{}$ is $\alpha_g$-strongly convex.
    \item[(ii)] $\Rq(\q):=\Rx(\x)+\Ry(\y)$ is an indicator of a compact and convex polytope $\F\equiv\{\q\in\E:\C\q\leq \bb\}$, where $\C:\E\to\R{p}$ is a linear mapping, and $\bb\in\R{p}$.
	\end{itemize}
\end{assumption}

\begin{theorem}\label{thm:second case}
    Suppose that Assumption \ref{asmp:polytope assumption} holds true. Then, there exists a constant $\sigma >0$ such that if $\rho\geq\alpha_g$, then Problem \eqref{pblm:problem} satisfies the PQG property with the parameter
    $\as=\alpha_g\sigma^{-1}$.
\end{theorem}
The proof is deferred to the  Appendix (see Section \ref{sec:properties proofs}).

\subsection{Weak Proximal Oracle}\label{sec:weak oracle section}
In this section, we present the main ingredient of our algorithm, which is used to update the primal variable $\q$ of Problem \eqref{pblm:reformulation} --- the weak proximal oracle, a concept which we adapt from \cite{garber2019fast} to our augmented Lagrangian framework.

To this end, we will need to define the following function. Given two points $\q\in\E$ and $\w\in\E_2$ together with two scalars $\eta\in(0,1]$ and $\mu>0$, we define the function
\begin{equation*}
    \begin{split}
        \Phi_\lambda(\vv)&:=\Rq(\vv)+\langle \vv,\nabla_\q S(\q,\w)+2\mu\K^{\top}\K\q\rangle+ \frac{\lambda}{2}\left(\eta(\bs+2\mu\|\K\|^2)\right)\|\vv-\q\|^2,
    \end{split}
\end{equation*}
where $\bs$ is the smoothness parameter of $S(\cdot,\w)$ (independent of $\w$), and $\lambda \geq 1$ is a parameter.

Before we formally define the notion of weak proximal oracle, we would like to define the notion of Strong Proximal Oracle. We say that a procedure is a \emph{(Strong) Proximal Oracle} applied to the augmented Lagrangian $\Lp(\q,\w)\equiv S(\q,\w)+\Rq(\q)$, which is associated with Problem \eqref{pblm:problem}, if it computes the exact minimizer of $\Phi_{\lambda}$ for $\lambda = 1$. That is, solves the problem $\min_{\vv\in\E}\Phi_1(\vv)$.
\begin{definition}\label{def:wpo def}(Weak Proximal Oracle)    
    We say that a procedure, which is denoted by WPO$_{\lambda}(\q,\w,\eta,\mu)$, is a \emph{$\lambda$-Weak Proximal Oracle} applied to $\Lp(\q,\w)$, if it returns a point $\vv\in\E$ which satisfies that
    \begin{equation}\label{eq:algorithm's wpo}
        \forall \q^*\in\Pp: \quad   \Phi_1(\vv)\leq \Phi_\lambda(\q^*),
    \end{equation}
     where we recall that $\Pp$ is the set of optimal solutions of Problem \eqref{pblm:problem}.
\end{definition}

Recalling that $\q$ is simply a convenient notation for the concatenation of the original two vector variables $\x$ and $\y$, the implementation of an oracle whose output satisfies \eqref{eq:algorithm's wpo} is naturally achieved by decoupling the condition \eqref{eq:algorithm's wpo} into two parts, one w.r.t. the variable $\x$ and the other w.r.t. the variable $\y$. That is, we consider two separate computations of two points, $\vv_\x\in\E_1$ and $\vv_\y\in\E_2$, satisfying the following inequalities  with some $\lambda_\x,\lambda_\y\geq1$:
\begin{equation}\label{eq:x wpo}
    \forall \x^*\in\X^*:\quad \Phi_1^\x(\vv_\x)\leq \Phi_{\lambda_\x}^\x(\x^*),
\end{equation}
where $\X^*:=\{\x^*\in\E_1:(\x^*,\A\x^*)\in\Pp\}$, and
\begin{equation}\label{eq:Jx def}
    \begin{split}
        \Phi_{\lambda_\x}^\x&:=\Rx(\vv_\x)+\langle \vv_\x,\nabla_\x S(\q,\w)+2\mu\A^{\top}\K\q\rangle+ \lambda_{\x}\frac{\eta(\bs+2\mu\|\K\|^2)}{2}\|\vv_\x-\x\|^2.
    \end{split}
\end{equation}
Similarly, 
\begin{equation}\label{eq:y wpo}
    \forall \y^*\in\Y^*:\quad  \Phi_1^\y(\vv_\y)\leq \Phi_{\lambda_\y}^\y(\y^*),
\end{equation}
where $\Y^*:=\{\A\x^*\in\E_2:\x^*\in\X^*\}$, and
\begin{equation}\label{eq:Jy def}
    \begin{split}
        \Phi_{\lambda_\y}^\y(\vv_\y)&:=\Ry(\vv_\y)+\langle \vv_\y,\nabla_\y S(\q,\w)-2\mu\K\q\rangle+ \lambda_{\y}\frac{\eta(\bs+2\mu\|\K\|^2)}{2}\|\vv_\y-\y\|^2.
    \end{split}
\end{equation}
The following proposition is a simple observation.
\begin{proposition}
    Assume that $\vv_\x\in\E_1$ satisfies \eqref{eq:x wpo} with some parameter $\lambda_\x\geq1$, and $\vv_\y\in\E_2$ satisfies \eqref{eq:y wpo} with some parameter $\lambda_\y\geq1$. Then,  $\vv=(\vv_\x,\vv_\y)$ satisfies \eqref{eq:algorithm's wpo} with $\lambda=\max\{\lambda_\x,\lambda_\y\}\geq1$.   
\end{proposition}
The main difficulty in satisfying \eqref{eq:x wpo} and \eqref{eq:y wpo} lies in the nonsmooth functions $\Rx$ and $\Ry$, respectively. This motivates the following definition.
\begin{definition}(Weak Proximal Friendly)\label{def:WPF}
    We say that a convex function $\Rx$ ($\Ry$) is \emph{weak proximal friendly} for Problem \eqref{pblm:problem}, if a point $\vv_\x\in\E_1$ ($\vv_\y\in\E_2$) satisfying \eqref{eq:x wpo} (\eqref{eq:y wpo}) (for some finite $\lambda_{\x}$ ($\lambda_{\y}$)) can be computed efficiently. We say that Problem \eqref{pblm:problem} is \emph{weak proximal friendly} if $\Rx$ and $\Ry$ are both weak proximal friendly.
\end{definition}
Let us now discuss some of the most important and interesting examples of weak proximal friendly functions.
\subsubsection{Proximal friendly functions}\label{sec:proximal friendly regularization}
When $\Rx$ $(\Ry)$ is a (strong) proximal friendly function (i.e., an exact minimizer of $\Phi^{\x}_1$ ($\Phi^{\y}_1$), as defined in \eqref{eq:Jx def} (\eqref{eq:Jy def}) could be computed efficiently), it follows immediately that it is also a weak proximal friendly function with parameter $\lambda_{\x}=1$ ($\lambda_{\y}=1$).

\subsubsection{Matrix nuclear norm regularization/constraint}\label{sec:matrices}    
In a typical low-rank matrix recovery setup, in which the nuclear norm is used as a convex surrogate for low-rank (see, for instance, the seminal works \cite{candes2012exact, candes2011robust}), we have that $\E_1 = \mathbb{R}^{m\times n}$ and $\Rx$ is an indicator function of a nuclear norm ball or a nuclear norm regularizer, or an indicator function of the spectrahedron (the set of all positive semidefinite matrices with trace equals some fixed positive parameter),  in case the solution is also required to be positive semidefinite. Assuming there exists a unique optimal low-rank solution, i.e., $\X^* = \{\XX^*\}$ with $\textrm{rank}(\XX^*) = k << \min\{m,n\}$, then an oracle for \eqref{eq:x wpo} amounts to computing a single rank-$k$ SVD of an $m\times n$ matrix, plus additional computationally-cheaper operations, which in high dimension is far more efficient than a proximal/projection computation, which in general requires a full-rank SVD, see detailed discussions in  \cite{allen2017linear, garber2018fast, garber2019fast}. Concretely, to satisfy \eqref{eq:x wpo} in this case we solve:
\begin{equation}\label{eq:low rank prox}
    \argmin_{\VV\in\mathbb{R}^{m\times n}:~\textrm{rank}(\VV)\leq k}\Phi_1^{\x}(\VV),
\end{equation}
and the corresponding WPO parameter is $\lambda_\x=1$.

Note that Problem \eqref{eq:low rank prox} follows the same structure of the standard proximal computation w.r.t. the function $\Rx$, only that  it is further constrained over the set of bounded rank matrices (which makes the problem more efficient to solve). Formally, we have the following theorem (extracted from the relevant discussions in \cite{allen2017linear, garber2018fast, garber2019fast}).
\begin{theorem}\label{matrices implementation}
     Let $(\QQ,\WW)\in \E\times \E_2$ be a pair of primal and dual points, where $\QQ = (\XX, \YY)\in\mathbb{R}^{m\times n}\times\E_2$. Let $\hat{\beta}:=\bs+2\mu\|\K\|^2$. Denote $\MM:=\XX-\frac{1}{\eta\hat{\beta}}\left(\nabla_\XX S(\QQ,\WW)+2\mu\A^{\top}\K\QQ\right).$ Let $\UU,\boldsymbol\Sigma,\VV$ be the SVD matrices of $\MM$,  i.e., $\MM=\UU\boldsymbol\Sigma \VV^{\top}$.
     \begin{itemize}
         \item If $\Rx(\cdot)=\nu\|\cdot\|_{\textrm{nuc}}$, for some $\nu>0$, then a solution to \eqref{eq:low rank prox} is the matrix $\tilde{\MM}=\UU\Tilde{\boldsymbol\Sigma}\VV^{\top}$, for 
        \begin{equation*}
            \Tilde{\boldsymbol\Sigma}=\textrm{diag}\left(\max\{\sigma_1(\MM)-\zeta,0\},\dots,\max\{\sigma_k(\MM)-\zeta,0\},0,\dots,0\right),
        \end{equation*}
         where $\zeta := \frac{\nu}{\eta\hat{\beta}}$.
         \item If $\Rx(\cdot)=\delta_{NB(\tau)}(\cdot)$ is the indicator function for the nuclear norm ball of radius $\tau$, then a solution to \eqref{eq:low rank prox} is given by $\tilde{\MM}=\UU\Tilde{\boldsymbol\Sigma}\VV^{\top}$, where here $\Tilde{\boldsymbol\Sigma}$ is the diagonal matrix whose diagonal is the projection of the vector $(\sigma_1(\MM),\dots,\sigma_k(\MM),0,\dots,0)$ onto the $\ell_1$ norm ball of radius $\tau$.
         \item If $\Rx(\cdot)=\delta_{\s}(\cdot)$ is the indicator function for the spectahedron $\{\XX~|~\XX\succeq 0,~\textrm{tr}(\XX)=\tau\}$, for some given $\tau > 0$, then letting $\UU\Lambda \UU^{\top}$ be the eigen-decomposition of $\MM$, a solution to \eqref{eq:low rank prox} is the matrix $\tilde{\MM}=\UU\Tilde{\Lambda} \UU^{\top}$, where $\tilde{\Lambda}$ is the diagonal matrix whose diagonal is the projection of the vector $(\lambda_1(\MM),\dots,\lambda_k(\MM),0,\dots,0)$ onto the simplex of radius $\tau$.
     \end{itemize}
In all these three cases, the runtime to compute $\tilde{\MM}$ is dominated by the computation of the top $k$ components in the SVD of $\MM$.     
\end{theorem}
  
\subsubsection{Polytope constraint}\label{sec:polytopes}
In case $\Rx$ ($\Ry$) is an indicator function for a convex and compact polytope for which a linear minimization oracle can be implemented efficiently, then $\Rx$ ($\Ry$) is also weak proximal friendly. Concretely, it is possible to construct an oracle for \eqref{eq:x wpo} (\eqref{eq:y wpo}) using a single call to the linear minimization oracle, plus some additional computations (that do not require any oracle access).
\begin{theorem}\label{thm:polytope implementaion}
    Let $\Rx$ $(\Ry)$ be an indicator function for a convex and compact polytope $\F$. Suppose a point $\x\in\E_1=\R{d}$ $(\y\in\E_2:=\R{d})$ is given explicitly as a convex combination of $t$ vertices of the polytope $\{\z_1,\dots,\z_t\}$. Let us denote the new output of the LMO of $\F$ w.r.t. the linear objective function determined by the vector $\p_{\x}:=\nabla_\x S(\q,\w)+2\mu\A^{\top}\K\q$ ($\p_{\y}:=\nabla_\y S(\q,\w)-2\mu\K\q$) by $\z_{t+1}$ and let $\MM=[\z_1,\dots,\z_{t+1}]\in\RR{d}{(t+1)}$. Let $\hat{\beta}:=\bs+2\mu\|\K\|^2$. Then, we can compute a point satisfying \eqref{eq:x wpo} (\eqref{eq:y wpo}) with a parameter $\lambda(\F)\geq1$ by returning the point $\vv_\x=\MM\boldsymbol\gamma^*_\x$ ($\vv_\y=\MM\boldsymbol\gamma^*_\y$), where $\boldsymbol\gamma^*_\x$ ($\boldsymbol\gamma^*_\y$) is an optimal solution to the following convex quadratic problem over the simplex of size $t+1$:
    \begin{equation}\label{pblm:x simplex}
        \min_{\boldsymbol\gamma\geq0,\langle\textbf{1},\boldsymbol\gamma\rangle=1}\langle \MM\boldsymbol\gamma,\p_{\x}\rangle+\frac{\eta\hat{\beta}}{2}\|\MM\boldsymbol\gamma-\x\|^2
    \end{equation}
    \begin{equation}\label{pblm:y simplex}
        \left(\min_{\boldsymbol\gamma\geq0,\\\langle\textbf{1},\boldsymbol\gamma\rangle=1}\langle \MM\boldsymbol\gamma,\p_{\y}\rangle+\frac{\eta\hat{\beta}}{2}\|\MM\boldsymbol\gamma-\y\|^2\right),
    \end{equation}
    where $\textbf{1}\in\R{t+1}$ is the vector of ones.
    
    Note that the returned solution $\vv_\x=\MM\boldsymbol\gamma^*_\x$ ($\vv_\y=\MM\boldsymbol\gamma^*_\y$) is now given explicitly in the form of a convex combination of at most $t+1$ vertices of the polytope.
\end{theorem}
The proof  is based on several observations from \cite{garber2016linearly} and is given in the Appendix (see Section \ref{sec:polytope oracle proof}).

The fact that the implementation of the WPO described in Theorem \ref{thm:polytope implementaion} only increases the number of vertices in the support of the computed point $\vv_\x$ ($\vv_\y$) by at most one is important, since our algorithm for solving the saddle point problem \eqref{pblm:reformulation} using this WPO makes a single call to this oracle per iteration. Hence, when the overall number of iterations is not very large, we will have that on each iteration, the input point $\x$ ($\y$) to this oracle will be supported on only a few vertices of the polytope, which means that Problem \eqref{pblm:x simplex} (\eqref{pblm:y simplex}) could be solved very efficiently.

The WPO parameter $\lambda(\F)$ depends on the geometry of the polytope $\F$ and may depend in worst case on the dimension, more details can be found in the Appendix (see Section \ref{sec:polytope oracle proof}).


\section{Some Illustrative Examples}\label{sec:Illustrative Examples}
In this section, we describe several families of problems of interest for which our two main assumptions --- the primal quadratic gap property (Definition \ref{def:primal quadratic gap}) and the availability of an efficient weak proximal oracle (Definition \ref{def:wpo def}), hold true.

\subsection{Structured Low-Rank Matrix Recovery}\label{sec:examples:matrix}
Consider the following optimization problem in which the goal is to recover a structured low-rank matrix from some noisy matrix observation $\boldsymbol\Sigma$:
\begin{equation}\label{pblm:matrix problem}
        \min_{\Ss\in\mathbb{R}^{m\times n}} \frac{1}{2}\|\Ss-\boldsymbol\Sigma\|_F^2 + \nu\Vert{\Ss}\Vert_{\textrm{nuc}}+\Ry(\A(\Ss)),\\
\end{equation}
where $\nu > 0$, $\Ry$ is assumed to be a proximal friendly function (or even a weak proximal friendly function, recall Definition \ref{def:WPF}), and the \emph{unique} optimal solution (guaranteed from strong convexity of the squared Frobenius norm term) is \emph{low rank}.

For instance, when $\A:=\I$ is the identity mapping and $\Ry$ is an $\ell_1$ regularizer, Problem \eqref{pblm:matrix problem} is a natural convex relaxation for recovering a matrix that is both low-rank and sparse, see for instance \cite{richard2012estimation}.

Setting $f(\Ss):=\frac{1}{2}\|\Ss-\boldsymbol\Sigma\|_F^2$, and $\Rx:=\nu\Vert{\Ss}\Vert_{\textrm{nuc}}$ we get a problem of the form of model \eqref{pblm:problem}. Primal quadratic gap holds since $f$ is strongly convex (see Theorem \ref{thm:first case}). By the assumption of low-rank of the optimal solution, as well as its uniqueness, according to the discussion in Section \ref{sec:matrices}, $\Rx$ is indeed a weak proximal friendly function and in general, Problem \eqref{pblm:matrix problem} is weak proximal friendly.

\subsection{Low-Rank Tensor Recovery}
The previous example could be extended to the more general and important problem of recovering low-rank tesnors. Analogously to matrices, the rank of a $N$-way real tensor $\XX$, $N > 2$, could be defined as the minimum number of rank one tensors (of the same dimensions) whose sum equals the original tensor, and here we denote it by $\textrm{rank}(\XX)$. However, in general, even determining the rank of a given tensor is NP-Hard \cite{johan1990tensor}. An alternative is to use standard matrix rank in an appropriate manner. Let $\XX\in\R{n_1\times\dots\times n_N}$ be a $N$-way tensor. For any $i\in\{1,\dots,N\}$, let $\A_i:\R{n_1\times\dots\times n_N}\rightarrow\RR{n_i}{I_i}$,  $I_i:=\frac{1}{n_i}\prod_{j=1}^Nn_j$, be a linear mapping that flattens the tensor into a $\RR{n_i}{I_i}$ matrix, by flattening all dimensions except for the $i$th dimension, see exact definition in  \cite{kolda2009tensor, gandy2011tensor}. This leads to the definition of the $n$-rank of a $N$-way tensor $\XX$, which is given by $\textrm{rank}_n(\XX):=(\textrm{rank}(\A_1\XX),\dots,\textrm{rank}(\A_N\XX))\in\mathbb{N}^N$, where $\textrm{rank}(\A_i\XX)$ is the standard matrix rank of $\A_i\XX$.  It is a fairly simple observation that $\max_{i\in\{1,\dots,N\}}\textrm{rank}_n(\XX)(i) \leq \textrm{rank}(\XX)$. This leads to the following natural convex relaxation \cite{gandy2011tensor}, which is analogous to nuclear norm-based relaxations for the matrix case, for the problem of recovering a low-rank tensor from a given noisy tensor measurement $\TT$:
\begin{equation}\label{pblm:tensor problem}
    \begin{split}
        \min_{\XX\in\R{n_1\times\dots\times n_N}} &\frac{1}{2}\|\XX-{\TT}\|_2^2+\Rx(\XX)+\nu\sum_{i=1}^N\|\YY_i\|_{\textrm{nuc}}\\
         \textrm{s.t. }& \A_i\XX=\YY_i~\textrm{ }\forall i\in[N],
    \end{split}
\end{equation}
where $\Vert{\cdot}\Vert_2$ denotes the $\ell_2$ norm for the appropriate tensor space, $\nu >0$, and $\Rx$ is some (weak) proximal friendly function that may encode additional structure of the tensor to be recovered (e.g., sparsity if we take it to be an $\ell_1$ regularizer).

Using the notations of Problem \eqref{pblm:problem}, we will denote the following $f(\XX)=\frac{1}{2}\|\XX-{\TT}\|_2^2$, $\A\XX:=[\A_1\XX^{\top},\dots,\A_N\XX^{\top}]^{\top}$, $\YY=\YY_1\times\dots\times \YY_N\in\mathbb{R}^{n_1\times I_1}\times\cdots\times\mathbb{R}^{n_N\times{}I_N}$, and $\Ry(\YY):=\nu\sum_{i=1}^N\|\YY_i\|_{\textrm{nuc}}.$

As in the previous example, the primal quadratic gap property holds since $f$ is strongly convex and the optimal solution $(\XX^*, \YY^*)$ is unique. Now, let us assume that $\XX^*$ has a low $n$-rank, meaning $\textrm{rank}(\YY_i^*) =  \textrm{rank}(\A_i\XX^*) \leq k$, for a fairly small $k$. In this case, since the variable $\YY$ is given as a cartesian product, the weak proximal computation w.r.t. the variable $\YY$ naturally decouples into $N$ separate weak proximal computation w.r.t. each of the blocks $\YY_1,\dots,\YY_N$. It should be noted that each block enjoys the same structure as in the low-rank matrix case\footnote{\cite{gandy2011tensor} formally shows this decoupling for strong proximal computations. It then becomes trivial to apply this for weak proximal computations.}, i.e., amounts to a $k$-SVD computation of a real matrix (as described in Section \ref{sec:proximal friendly regularization}). Since additionally $\Rx$ is assumed to be a (weak) proximal friendly function, we have that in general, Problem \eqref{pblm:tensor problem} is weak proximal friendly.

\subsection{Least Squares Over Intersection of Polytopes}
Let $\F_1,\dots,\F_n$ be convex and compact polytopes in $\mathbb{R}^d$, for which a linear minimization oracle can be implemented efficiently. Given $\MM\in\RR{p}{d}$ and $\bb\in\R{p}$, we consider the following constrained least squares optimization problem:
\begin{equation}\label{pblm:least squares problem}
        \min_{\x\in\R{d}} \frac{1}{2}\|\MM \x-\bb\|^2 \quad
        \textrm{s.t.} \quad \x\in\bigcap_{i=1}^n\F_i.
\end{equation}

We can see that this problem is of the form of \eqref{pblm:problem} by setting $f(\x):=\frac{1}{2}\|\MM \x-\bb\|^2$, $\Rx(\x):=\delta_{\F_1}(\x)$, $\A:=[\I,\dots,\I]^\top$ ($n-1$ times), where $\I$ is the identity mapping, and $\Ry(\y)=\Ry([\y_1^{\top},\dots,\y_{n-1}^{\top}]^{\top}):=\sum_{i=1}^{n-1}\delta_{\F_{i+1}}(\y_i)$, where $\delta_{\F_i}(\cdot)$ is the indicator function of the polytope $\F_i$, $i=1,\dots,n$.

Notice we can write $f$ as $g\circ\B$, where $\B=\MM$, and $g(\cdot)=\frac{1}{2}\|(\cdot)-\bb\|^2$, which is a 1-strongly convex function. In addition, $\Rq$ here is an indicator function of the product of the polytopes $\F_1\times\cdots\times\F_n$. Hence, according to Theorem \ref{thm:second case}, by setting the augmented Lagrangian parameter to $\rho\geq1$, this problem satisfies the primal quadratic gap property.

Since we assumed each of the polytopes $\F_i$ admits an efficient linear minimization oracle, as discussed in Section \ref{sec:polytopes}, each of the indicator functions $\delta_{\F_i}$ is weak proximal friendly. Further more, since the variable $\y$ admits a simple cartesian product structure, the weak proximal oracle w.r.t. $\y$ naturally decouples into $n-1$ separate weak proximal oracle computations, each w.r.t. to one of the polytopes $\F_i$, $i=2,\dots,n$, and so, Problem \eqref{pblm:least squares problem} is also weak proximal friendly.

\section{Algorithm and Convergence Analysis}\label{sec:Algorithm & Analysis}
Our weak proximal oracle-based algorithm for solving the saddle point problem \eqref{pblm:reformulation} is given as Algorithm \ref{alg:Alg1}. At each iteration, the primal variable is updated using a Frank-Wolfe-style update, in the sense that the updated primal variable is given as a convex combination of the previous primal iterate and the output of some oracle. As opposed to the classical Frank-Wolfe method, which relies on the output of a linear minimization oracle, here we rely on the output of a weak proximal oracle described in the preceding sections. The update of the dual variable is done via a standard gradient ascent step.

\begin{algorithm}[H]
        \begin{enumerate}
            \item \textbf{Input:} Primal and dual step sizes $\eta\in[0,1], \mu>0$, a WPO oracle with parameter $\lambda\geq1$, and initialization $(\q_0,\w_0)\in \textrm{dom}(\Rq)\times\E_2$. 
            \item \textbf{Main step:} For $t = 0,1,...$ generate the sequence $\{(\q_t,\w_t)\}_{t\in\N}$ as follows
            \begin{align}
                \label{eq:WP}
                    &\vv_t = \textrm{WPO}_{\lambda}(\q_t,\w_t,\eta,\mu),\\
                \label{eq:primal step}
                    &\q_{t+1} = (1-\eta)\q_t+\eta\vv_t,\\
                \label{eq:dual step}
                    &\w_{t+1} = \w_t+\mu \K\q_{t+1}.
            \end{align}
        \end{enumerate}
    \caption{Weak Proximal Method of Multipliers}
    \label{alg:Alg1}
\end{algorithm}

Before stating our main result, the convergence guarantees of Algorithm \ref{alg:Alg1}, let us  introduce some helpful notation.
\begin{itemize}
    \item For any $\q=(\x,\y)\in\E$, we denote the value of the objective function of Problem \eqref{pblm:problem} by $h(\q):=f(\x)+\Rx(\x)+\Ry(\y)$.
    \item Given a sequence $\{\q_t\}_{t\in\N}$, we define the corresponding ergodic sequence $\{{\bar \q}_t\}_{t\in\N}$ by $\bar{\q}_t := \frac{1}{T}\sum_{j=0}^{t-1}\q_{j+1}$ for all $t \in \N$.
    \item For any $t\in\N$, we denote by $d_t:=\Lp(\q_t,\w_t)-\Lp^*$, where $\Lp^*$ denotes the AL value of a saddle point.
\end{itemize}
To formally present our main result, we also recall that $\alpha_{S}$ and $\beta_{S}$ denote the primal quadratic gap parameter and smoothness parameter of $S(\q,\w)$, respectively.
\begin{theorem}(An O(1/T) ergodic convergence rate)\label{thm:main}
    Let $(\q^*,\w^*)$ be a saddle point of $\Lp$ and let $\{(\q_t,\w_t)\}_{t\in\N}$ be a sequence generated by Algorithm \ref{alg:Alg1} with a primal step size $\eta=\frac{\as}{2\lambda(\bs+2\mu(\|\A\|+1)^2)}\in(0,1],$ and $\mu$ satisfying $0<\mu \leq \frac{\sqrt{\lambda\as^2+\lambda^2\bs^2}-\lambda\bs}{4\lambda(\|\A\|+1)^2}$. Then, for any $c\geq2\|\w^*\|$ and any integer $T\geq1$ we have
    \begin{align*}
     h(\bar{\q}_T)-h(\q^*)\leq\frac{B(\rho,\mu)}{T} ~~~ \textrm{and}~~~ \|\K\bar{\q}_T\|\leq\frac{2B(\rho,\mu)}{cT},
    \end{align*}
    where
    \begin{equation*}
        B(\rho,\mu) = \frac{(c+\|\w_0\|)^2}{2\mu}+\max\left\{0,\frac{2d_1(\beta+(\rho+2\mu)(\|\A\|+1)^2)}{\as}\right\},
    \end{equation*}
    $d_1=\Lp(\q_1,\w_1)-\Lp(\q^*,\w^*)$ and $\beta$ is the smoothness parameter of $f$.
\end{theorem}

In order to prove Theorem \ref{thm:main}, we first need to establish two intermediate results. The first result was established in \cite{sabach2019lagrangian}. For the sake of completeness we state it here and reprove it. 
\begin{lemma}\label{lemma:objective and feasibility convergence}
    (Objective and feasibility approximation)
    Let $(\q^*,\w^*)$ be a saddle point of $\Lp$. Let $\q\in\E$, and suppose that $c\geq2\|\w^*\|$, for some $c>0$. If
    \begin{equation}\label{eq:obj feas approx}
        h(\q)-h(\q^*)+c\|\K\q\|+\frac{\rho}{2}\|\K\q\|^2\leq\delta,
    \end{equation}
    holds for some $\delta\geq0$, then the following assertions hold
    \begin{enumerate}
        \item\label{eq:obj boun} $h(\q)-h(\q^*)\leq\delta$.
        \item\label{eq:feas bound} $\|\K\q\|\leq\frac{2\delta}{c}$.
    \end{enumerate}
\end{lemma}
\begin{proof}
    The first assertion holds since $c\|\K\q\|+\frac{\rho}{2}\|\K\q\|^2\geq0$. Moreover, since $(\q^*,\w^*)$ is a saddle point of $\Lp$, we have    
    \begin{equation}\label{eq:saddle ineq}
            h(\q^*)=\Lp(\q^*,\w^*)\leq\Lp(\q,\w^*)=h(\q)+\langle\w^*,\K\q\rangle+\frac{\rho}{2}\|\K\q\|^2.
    \end{equation}
    Combining \eqref{eq:obj feas approx} and \eqref{eq:saddle ineq} we have    
    \begin{equation*}
        c\|\K\q\|\leq\delta+\langle\w^*,\K\q\rangle\leq\delta+\|\w^*\|\cdot\|\K\q\|\leq\delta+\frac{c}{2}\|\K\q\|.
    \end{equation*}
    Rearranging the last inequality, yields the second assertion.
\end{proof}

The following lemma accounts for the main step in the proof of our fast convergence rates given in Theorem \ref{thm:main} and also accounts for the main technical novelty of our paper. It leverages both the primal quadratic gap property and the weak proximal oracle, to obtain a linear convergence rate for the value of the augmented Lagrangian.
\begin{lemma}\label{lemma:linear rate}
    Let $\{(\q_t,\w_t)\}_{t\in\N}$ be a sequence generated by Algorithm \ref{alg:Alg1} with $\eta$ and $\mu$ as given in Theorem \ref{thm:main}. Then, for all $t \in \N$ we have that $d_{t+1}\leq(1-\eta)^td_1$.
\end{lemma}
\begin{proof}
    For any $t\in\N$, let $(\q_t^*,\w^*)$ be a saddle point of $\Lp$, such that the pair $\{\q_t,\q_t^*\}$ satisfies the primal quadratic gap (see \eqref{eq:pqg}), i.e., $\q_t^*=\argmin_{\q^*\in\Pp}\|\q_t-\q^*\|^2$ and $\w^*\in\D$. Since $\Lp(\q_t^*,\w^*)=\Lp^*$, we have    
    \begin{align}
            d_{t+1}&=\Lp(\q_{t+1},\w_{t+1})-\Lp(\q_t^*,\w^*) \nonumber\\
            &=\Lp(\q_{t+1},\w_{t+1})-\Lp(\q_{t+1},\w_t)+\Lp(\q_{t+1},\w_t)-\Lp(\q_t,\w_t)+\Lp(\q_t,\w_t)-\Lp(\q_t^*,\w^*) \nonumber\\
            &=\langle\w_{t+1}-\w_t,\K\q_{t+1}\rangle+\Lp(\q_{t+1},\w_t)-\Lp(\q_t,\w_t)+d_t \nonumber\\
            &=d_t+\mu\|\K\q_{t+1}\|^2+\Lp(\q_{t+1},\w_t)-\Lp(\q_t,\w_t), \label{eq:dt+1}
    \end{align}
    where the last equality follows from \eqref{eq:dual step}. Now, using the feasibility of $\q_t^*$ ($\K\q_t^* = 0$), we have    
    \begin{align}
            \|\K\q_{t+1}\|^2 & =\|\K\q_{t+1}-\K\q_t+\K\q_t-\K\q_t^*\|^2 \nonumber\\
            &=\|\K\q_{t+1}-\K\q_t\|^2+2\langle \K\q_{t+1}-\K\q_t,\K\q_t-\K\q_t^*\rangle+\|\K\q_t-\K\q_t^*\|^2 \nonumber\\
            &=\eta^2\|\K\vv_t-\K\q_t\|^2+2\eta\langle \K\vv_t-\K\q_t,\K\q_t\rangle+\|\K\q_t-\K\q_t^*\|^2 \nonumber\\
            &\leq \eta^2\|\K\|^2\cdot\|\vv_t-\q_t\|^2+\eta\langle\vv_t-\q_t,2\K^{\top}\K\q_t\rangle+\|\K\|^2\cdot\|\q_t-\q_t^*\|^2,\label{eq:rho part}
    \end{align}
    where the last equality follows from \eqref{eq:primal step}. By the $\bs$-smoothness of $S(\cdot,\w)$, the convexity of $\Rq(\q)$ and \eqref{eq:primal step}, we have    
    \begin{align}
        \Lp(\q_{t+1},\w_t)-\Lp(\q_t,\w_t)& =S(\q_{t+1},\w_t)-S(\q_t,\w_t)+\Rq(\q_{t+1})-\Rq(\q_t) \nonumber\\
        & \leq\frac{\eta^2\bs}{2}\|\vv_t-\q_t\|^2+\eta\langle\vv_t-\q_t,\nabla_\q S(\q_t,\w_t)\rangle+\eta(\Rq(\vv_t)-\Rq(\q_t)).\label{eq:beta part}
    \end{align}
    Now, from \eqref{eq:WP}, recalling \eqref{eq:algorithm's wpo}, since $\q_t^*\in\Pp$, we have that
    \begin{align}
            &\Rq(\vv_t)+\langle \vv_t,\nabla_\q S(\q_t,\w_t)+2\mu\K^{\top}\K\q_t\rangle+ \frac{\eta(\bs+2\mu\|\K\|^2)}{2}\|\vv_t-\q_t\|^2\leq \nonumber \\
            &\Rq(\q_t^*)+\langle\q_t^*,\nabla_\q S(\q_t,\w_t)+2\mu\K^{\top}\K\q_t\rangle + \frac{\lambda\eta(\bs+2\mu\|\K\|^2)}{2}\|\q_t^*-\q_t\|^2.\label{eq:wpo use}
    \end{align}
    By combining \eqref{eq:dt+1}, \eqref{eq:rho part}, \eqref{eq:beta part} and \eqref{eq:wpo use} we obtain the following   
    \begin{align}
            d_{t+1}-d_t & =\mu\|\K\q_{t+1}\|^2+\Lp(\q_{t+1},\w_t)-\Lp(\q_t,\w_t) \nonumber\\
            & \leq \frac{2\mu\|\K\|^2}{2}\|\q_t-\q_t^*\|^2+\frac{\eta^2(\bs+2\mu\|\K\|^2)}{2}\|\vv_t-\q_t\|^2+\eta(\Rq(\vv_t)-\Rq(\q_t)) \nonumber\\
            &+\eta\langle\vv_t-\q_t,\nabla_\q S(\q_t,\w_t)+2\mu\K^{\top}\K\q_t\rangle \nonumber\\
            &\leq\frac{1}{2}\left(2\mu\|\K\|^2+\lambda\eta^2(\bs+2\mu\|\K\|^2)\right)\|\q_t-\q_t^*\|^2+\eta(\Rq(\q_t^*)-\Rq(\q_t)) \nonumber\\
            & + \eta\langle\q_t^*-\q_t,\nabla_\q S(\q_t,\w_t)\rangle+2\mu\eta\langle \K\q_t^*-\K\q_t,\K\q_t\rangle.\label{eq:weak oracle}
    \end{align}    
    Let us, for convenience, denote $r=2\mu(\|\A\|+1)^2\geq2\mu\|\K\|^2$. Now, recall that $\q_t^*$ was chosen to satisfy \eqref{eq:pqg}. Therefore, from \eqref{eq:weak oracle} and \eqref{eq:pqg}, as well as the feasibility of $\q_t^*$, we obtain that    
    \begin{align}
            d_{t+1}-d_t\leq&\frac{1}{2}\left(2\mu\|\K\|^2-\eta\as+\lambda\eta^2(\bs+2\mu\|\K\|^2)\right)\|\q_t-\q_t^*\|^2-2\mu\eta\|\K\q_t\|^2\\+&\eta[S(\q_t^*,\w_t)+\Rq(\q_t^*)-S(\q_t,\w_t)-\Rq(\q_t) \nonumber]\\
            \leq&\frac{1}{2}\left(r-\eta\as+\lambda\eta^2(\bs+r)\right)\|\q_t-\q_t^*\|^2+\eta[\Lp(\q_t^*,\w_t)-\Lp(\q_t,\w_t)].\label{eq:pqg use}
    \end{align}
    Since $\q_t^*$ is a feasible solution of the Problem \eqref{pblm:problem}, we have that $\Lp(\q_t^*,\w_t)=\Lp(\q_t^*,\w^*)=\Lp^*$. Thus, from \eqref{eq:pqg use}, after substituting $\eta=\frac{\as}{2\lambda(\bs+r)}\in[0,1]$, we have for all $t \in \N$ (recall that $\lambda \geq 1$)
    \begin{equation*}
            d_{t+1}\leq(1-\eta)d_t+\frac{1}{2}(r-\eta\as+\lambda\eta^2(\bs+r))\|\q_t-\q_t^*\|^2 =(1-\eta)d_t+\frac{1}{2}\left(r-\frac{\as^2}{4\lambda(\bs+r)}\right)\|\q_t-\q_t^*\|^2.
    \end{equation*}
    Taking any $\mu>0$ such that $2\mu(\|\A\|+1)^2\equiv r\leq\frac{\sqrt{\lambda\as^2+\lambda^2\bs^2}-\lambda\bs}{2\lambda}$, we get that $r-\frac{\as^2}{4\lambda(\bs+r)}\leq0$, and therefore we immediately obtain that
    $$d_{t+1}\leq(1-\eta)d_t,$$
    which proves the desired result.
\end{proof}

We will now proceed to prove our convergence rate result.
\begin{proof}[Proof of Theorem \ref{thm:main}]
    First, let $\{\w_t\}_{t\in\N}$ be the sequence of dual points generated by Algorithm \ref{alg:Alg1}. For any $t\in\N$ and any $\w\in\E_2$, we define $$\Delta_t(\w) := \frac{1}{2}\|\w-\w_t\|^2-\frac{1}{2}\|\w-\w_{t+1}\|^2.$$
    \\
    We notice that for any $t\in\N$ and $\w\in\E_2$ we have
    \begin{align*}\label{eq:telescopic bound}
        \Lp(\q_{t+1},\w)-\Lp(\q_{t+1},\w_{t+1})&=\langle\w-\w_{t+1},\K\q_{t+1}\rangle\\&=\frac{1}{\mu}\langle\w-\w_{t+1},\w_{t+1}-\w_t\rangle\\&=\frac{1}{\mu}\Delta_t(\w)-\frac{1}{2\mu}\|\w_{t+1}-\w_t\|^2\\&\leq\frac{1}{\mu}\Delta_t(\w).
    \end{align*}
    Thus, summing the above inequality for all $t =0 , 1 , \ldots , T - 1$, we have
    \begin{align}
            \sum_{t=0}^{T-1}(\Lp(\q_{t+1},\w)-\Lp(\q_{t+1},\w_{t+1})) &\leq \frac{1}{\mu}\sum_{t=0}^{T-1}\Delta_t(\w) \nonumber\\
            &= \frac{1}{2\mu}\sum_{t=0}^{T-1}(\|\w-\w_t\|^2-\|\w-\w_{t+1}\|^2) \nonumber\\
            &=\frac{1}{2\mu}(\|\w-\w_0\|^2-\|\w-\w_T\|^2)\\
            &\leq \frac{1}{2\mu}\|\w-\w_0\|^2 \nonumber\\
            &\leq \frac{1}{2\mu}(\|\w\|+\|\w_0\|)^2.\label{eq:telescopic sum}
    \end{align}
    In addition, from Lemma \ref{lemma:linear rate}, we have that $d_{t+1}\leq d_1(1-\eta)^t$. Thus, if $d_{1} > 0$, we have that
    \begin{equation*}
        \sum_{t=0}^{T-1}(\Lp(\q_{t+1},\w_{t+1})-\Lp(\q^*,\w^*))=\sum_{t=0}^{T-1}d_{t+1}\leq d_1\sum_{t=0}^{T-1}(1-\eta)^t\leq d_1\sum_{t=0}^{\infty}(1-\eta)^t=\frac{d_1}{\eta},
    \end{equation*}
    where the second equality follows from the classical result on the geometric series. Hence, using Lemma \ref{lemma:loose upper bound for K}, we obtain that
    \begin{equation}\label{eq:geo seq sum d>0}
        \sum_{t=0}^{T-1}(\Lp(\q_{t+1},\w_{t+1})-\Lp(\q^*,\w^*))\leq\frac{2d_1}{\as}\left(\beta+(2\mu+\rho)(\|\A\|+1)^2\right).
    \end{equation}
    Otherwise, if $d_1\leq0$, we have    
    \begin{equation}\label{eq:geo seq sum d<=0}
            \sum_{t=0}^{T-1}(\Lp(\q_{t+1},\w_{t+1})-\Lp(\q^*,\w^*))=\sum_{t=0}^{T-1}d_{t+1}\leq\sum_{t=0}^{T-1}d_1(1-\eta)^t\leq0.
    \end{equation}
    Combining \eqref{eq:geo seq sum d>0} and \eqref{eq:geo seq sum d<=0} yields    
    \begin{equation}\label{eq:geo seq sum}
            \sum_{t=0}^{T-1}(\Lp(\q_{t+1},\w_{t+1})-\Lp(\q^*,\w^*))\leq\max\left\{0,\frac{2d_1}{\as}(\beta+(\rho+2\mu)(\|\A\|+1)^2)\right\}.
    \end{equation}
    Now, since $\Lp(\cdot,\w)$ is convex, and since $\Lp(\q^*,\w^*)=\Lp(\q^*,\w)$ for any $\w\in\E_2$, which follows from the feasibility of $\q^*$, we get, for any $\w\in\E_2$, that    
    \begin{align}
            h(\bar{\q}_T)-h(\q^*)+\langle\w,\K\bar{\q}_T\rangle+\frac{\rho}{2}\|\K\bar{\q}_T\|^2&=\Lp(\bar{\q}_T,\w)-\Lp(\q^*,\w) \nonumber\\
            &\leq\frac{1}{T}\sum_{t=0}^{T-1}(\Lp(\q_{t+1},\w)-\Lp(\q^*,\w)) \nonumber\\
            &=\frac{1}{T}\sum_{t=0}^{T-1}(\Lp(\q_{t+1},\w)-\Lp(\q^*,\w^*)) \nonumber\\
            &=\frac{1}{T}\left[\sum_{t=0}^{T-1}(\Lp(\q_{t+1},\w)-\Lp(\q_{t+1},\w_{t+1}))\right. \nonumber\\
            & +\left.\sum_{t=0}^{T-1}(\Lp(\q_{t+1},\w_{t+1})-\Lp(\q^*,\w^*))\right] \nonumber\\
            &\leq\frac{1}{T}\left[\frac{(\|\w\|+\|\w_0\|)^2}{2\mu} \right.\nonumber\\&\left.+ \max\left\{0,\frac{2d_1}{\as}(\beta+(\rho+2\mu)(\|\A\|+1)^2)\right\}\right],\label{eq:upper bound for Q bar}
    \end{align}
    where the last inequality follows from \eqref{eq:telescopic sum} and \eqref{eq:geo seq sum}. Maximizing both sides over $\|\w\|\leq c$, we get    
    \begin{equation*}
        \begin{split}
            h(\bar{\q}_T)-h(\q^*)+c\|\K\bar{\q}_T\|+\frac{\rho}{2}\|\K\bar{\q}_T\|^2&\leq\frac{1}{T}\left[\frac{(c+\|\w_0\|)^2}{2\mu}\right.\\&\left.+\max\left\{0,\frac{2d_1(\beta+(\rho+2\mu)(\|\A\|+1)^2)}{\as}\right\}\right]\\&=\frac{B(\rho,\mu)}{T}.
        \end{split}
    \end{equation*}
    The result now follows from Lemma \ref{lemma:objective and feasibility convergence} by substituting $\delta=B(\rho,\mu)/T$.
\end{proof}

\section{Numerical Experiments}\label{sec:Numerical Experiments}
We compare the empirical performance of our algorithm with the state-of-the-art projection-free conditional gradient-based algorithm from  \cite{yurtsever2019conditional}. For each of the algorithms we tested two variants, see Table \ref{tab:algorithms compared} and discussions in the sequel for details.

We consider two tasks which involve optimization with low-rank matrices and are cast as optimization over the spectrahedron: estimation of a low-rank and sparse covariance matrix from noisy observations, and the semidefinite relaxation for Max Cut. For both tasks, implementing the WPO in our algorithms involves a low-rank SVD with rank that is at least that of the optimal solution (see  Section \ref{sec:examples:matrix}). Thus, throughout this section we denote the rank of the optimal solution by $r^*$ and our algorithm's estimate of it by $\hat{r}^*$.

\subsection{Low-rank and sparse covariance  estimation}
We consider the following convex relaxation for recovering a low-rank and sparse positive semidefinite matrix from a noisy matrix observation $\widehat{\boldsymbol\Sigma}$:
\begin{equation}\tag{CME}\label{pblm:CME}
        \min_{\Ss}\frac{1}{2}\|\Ss-\widehat{\boldsymbol\Sigma}\|_F^2\\
       ~~ \textrm{s.t.}~~ \Ss\succeq0,\textrm{tr}(\Ss)=\tau,\|\Ss\|_1\leq s.
\end{equation}
As discussed in Section \ref{sec:examples:matrix}, this problem satisfies all the assumptions required for our theoretical guarantees to hold.

\paragraph*{Data generation.} Our experiment is inspired by previous experiments conducted in \cite{richard2012estimation} and \cite{gidel2018frank}. We first generate a block diagonal, sparse and low-rank covariance matrix $\boldsymbol\Sigma\in\RR{d}{d}$, where we set $d=400$. Then, we draw $d$ vectors $\z_i \sim \mathcal{N}(0, \boldsymbol\Sigma)$, add a Gaussian noise $\mathcal{N}(0,\sigma=0.6)$ to each entry of $\z_i$, and create the noisy matrix $\widehat{\boldsymbol\Sigma}:=\frac{1}{d}\sum_{i=1}^d \z_i\z_i^{\top}$. To create the blocks of $\boldsymbol\Sigma$, we use $r$ blocks of the form $\uu\uu^{\top}$ where $\uu \sim\mathcal{U}([-1,1])$. This way we ensure that $\boldsymbol\Sigma$ is of rank at most $r$. In order to enforce sparsity, while ensuring the low rank of $\boldsymbol\Sigma$, \emph{before} computing $\uu\uu^{\top}$, we only keep the entries $u_i$ for which $|u_i| > 0.9$ (the rest of the entries become zero). Our choices of $\tau$ and $s$, the radius of the nuclear norm ball and the $\ell_1$ norm ball, respectively, are chosen as the nuclear norm and $\ell_1$ norm of $\boldsymbol\Sigma$.  We use rank values $r\in\{5, 10, 20\}$.

\subsubsection{Implementation details}

\paragraph{Our algorithm.}
\begin{itemize}
    \item Implementing the algorithm, we set $\Rx$ and $\Ry$ to be indicators of the spectrahedron $\s$ and the $\ell_1$-norm ball of radius $s$, respectively, and we set $\A$ to be the identity operator. Hence in each iteration the main computations of our algorithm are one low-rank SVD and one $\ell_1$-norm ball projection. We also set $f(\x)\equiv f(\Ss):=\frac{1}{2}\|\Ss-\widehat{\boldsymbol\Sigma}\|_F^2$. In particular, we set the rank estimate for our algorithm, which determines the rank of the SVD computations, as $\hat{r}^* = r$.
    \item Given the values of $\rho$ and $\mu$, we set the value of the primal step size $\eta$ to be sent to the WPO according to the  formula given in Theorem \ref{thm:main}, with respect to the value of $\mu$, and the values of $\as$ and $\bs$ induced from the value of $\rho$ (notice here the WPO parameter is $\lambda=1$). Note that for any $\mu>0$, $\eta\in[0,1]$. 
    \item For the actual primal step, we used line search. That is, we set:
    $$\eta_t=\argmin_{\eta\in[0,1]}\{\mu\|\K\q_{t+1}(\eta)\|^2+\Lp(\q_{t+1}(\eta),\w_{t})\},$$ where $\q_{t+1}(\eta):=\eta \vv_t+(1-\eta)\q_t$, and $\vv_t$ is the output of the WPO. This problem has a closed form solution. It could be easily verified that using this line search does not change the gurantees of Theorem  \ref{thm:main}.
    \item We implemented and examined two variants of our algorithm. One which returns the mean of the sequence of primal points generated by Algorithm \ref{alg:Alg1}, which we dub "Mean". This is the variant for which we have a theoretically guaranteed convergence rate. The second variant, which we dub "Last", returns the last primal point of the sequence generated by Algorithm \ref{alg:Alg1}. This variant seems to be a natural improvement for the Mean variant in practice, in terms of convergence rate.
\end{itemize}

\paragraph{The baseline algorithm.}
The baseline algorithm for our experiments is the CGAL algorithm introduced in \cite{yurtsever2019conditional} (Algorithm 1 in the paper). This algorithm solves Problem \eqref{pblm:problem} if $\Rx$ and $\Ry$ are indicators of convex sets $\X\subset\E_1$, which is compact, and $\Y\subseteq\E_2$. \cite{yurtsever2019conditional} suggests two methods for computing the dual step size at every time step --- decreasing step sizes (which we dub "decr"), and constant step sizes (which we dub "const"). We present results for both choices. These dual updates also depend on a sequence of dual bounds $\{D_k\}_{k\in\N}$. In accordance with the recommendation in \cite{yurtsever2019conditional} for best practical performance, we set $D_k=D_\X\|\A\|\rho_0,\forall k\in\N$, where $D_\X$ is the diameter of $\X$ and $\rho_0$ is an initial penalty parameter for the augmented Lagrangian (this algorithm uses increasing values of $\rho$, while we use a constant value). In the implementation of this algorithm we also set $\A$ to be the identity operator, $\X$ to be the spectrahedron, and $\Y$ to be an $\ell_1$-norm ball. One iteration of CGAL, thus, requires a single rank-one SVD operation and two projections onto the $\ell_1$-norm ball.

\paragraph{SVD implementation and rank overestimation.}\label{sec:SVD}
Both algorithms require low-rank SVD computations. The baseline requires a single rank-one SVD on each iteration, and our algorithm performs a single rank-$\hat{r}^*$ SVD every iteration, where we recall $\hat{r}^*$ is an estimation of $r^*$, the rank of the optimal solution, which is assumed to be small. We perform these computations for both algorithms using the \emph{scipy.sparse.linalg.eigsh()} built-in function in Python, which gets as an input the rank of the SVD required. For our algorithm, we examined two cases. In the first one, we set $\hat{r}^*=r$. In the second case, we purposely overestimated $r^*$ by a factor of 1.5, meaning we set $\hat{r}^*=1.5r$. Another parameter of this function is a tolerance parameter, which serves as a stopping condition, which we set to 0.01 for both algorithms, to avoid long running times reaching irrelevant levels of precision.

\begin{table}[H]
    \centering
{\small
    \begin{tabular}{|c|p{10cm}|c|}
         \hline
         Algorithm & Description & SVD rank\\
         \hline
         Mean (ours) & Returns the mean (ergodic) of the primal sequence produced by Algorithm \ref{alg:Alg1} & $\hat{r}^*$\\
         \hline
         Last (ours) & Returns the last primal point produced by Algorithm \ref{alg:Alg1} & $\hat{r}^*$\\
         \hline
         CGAL-const & Baseline algorithm from \cite{yurtsever2019conditional} using "const" option for dual updates & 1\\
         \hline
         CGAL-decr & Baseline algorithm from \cite{yurtsever2019conditional} using "decr" option for dual updates & 1\\
         \hline
    \end{tabular}
    \caption{Description of algorithms compared in the empirical study. The last column specifies the rank of SVD needed to update primal variable at each iteration.}\label{tab:algorithms compared}}
    
\end{table}

\paragraph{Manual Tuning of hyper parameters.}
In our algorithm, given the value of the quadratic penalty parameter $\rho$, our theory suggests taking values of the dual step size $\mu$ which are often highly pessimistic in practice (less than $10^{-3}$ for this experiment). We thus increase the value of $\mu$ beyond its theoretical bound for better practical results and simply set $\mu=0.2$, which seems to work well. For the tuning of $\rho$, given the choice $\mu=0.2$, we started with 1, and kept multiplying/dividing by a factor of 5, until a parameter outperformed its neighbors (the neighbors of 1 are 0.2 and 5, to be clear), over the average of  20 i.i.d. runs. The tuning was done separately for the two variants of our algorithm ("Mean" and "Last"). The tuning of the parameter $\rho_0$ for the baseline algorithm was done similarly, and was done separately for its two different variants (``decr'' and ``const''). The values of $\rho$ for our algorithm and $\rho_0$ for the baseline are presentes in Table \ref{tab:hyper parameters}.

\begin{table}[H]
    \centering
    \begin{tabular}{|c|c|c|c|}
         \hline
         \backslashbox{Algorithm}{$r$} & 5 & 10 & 20\\
         \hline
         Last (ours) & 25 & 5 & 1\\
         \hline
         Mean (ours) & 5 & 5 & 1\\
         \hline
         CGAL-const & 1 & 0.2 & 0.2\\
         \hline
         CGAL-decr & 1 & 1 & 1\\
         \hline
    \end{tabular}
    \caption{Values of $\rho$ (our algorithm, top two rows) and $\rho_0$ (baseline algorithm, bottom two rows).}
    \label{tab:hyper parameters}
\end{table}

\paragraph{Initialization.}
For both algorithms we initialized $\Ss$ to be the projection of $\widehat{\boldsymbol\Sigma}$ onto the spectrahedron of trace $\tau$, and the dual variable to be zero. In addition, for our algorithm, we initialized the additional variable $\y$ to be the projection of $\widehat{\boldsymbol\Sigma}$ onto the $\ell_1$ norm ball of radius $s$.

\subsubsection{Results}

We ran both algorithms (two variants for each) for a fixed number of iterations $T=2000$. The results are the averages over 20 i.i.d. runs of the experiment (each experiment randomly selects $\boldsymbol\Sigma$ and $\widehat{\boldsymbol\Sigma}$). In the graphs, we plot the (normalized) objective $\frac{\|\Ss-\widehat{\boldsymbol\Sigma}\|_F^2}{2\|\widehat{\boldsymbol\Sigma}\|_F^2}$, the distance from feasibility $\|\Ss-P_{\ell_1(s)}(\Ss)\|_F$, where $P_{\ell_1(s)}(\Ss)$ is the projection of $\Ss$ onto the $\ell_1$-norm ball of radius $s$, and the recovery error of $\boldsymbol\Sigma$ measured by $\frac{\|\Ss-\boldsymbol\Sigma\|_F^2}{2\|\boldsymbol\Sigma\|_F^2}$. All measures are plotted both w.r.t. the number of iterations and  the runtime (in seconds). Initially we set the SVD rank parameter for our algorithm to be exactly $r$, i.e., we set $\hat{r}^* = r$. To simulate a more challenging and realistic setting, we redid the experiments  for $r=10$ and $r=20$, but this time we run our algorithm with a 1.5x overestimate of the rank, i.e., we use for it SVD computations of rank $\hat{r}^*=15$ and $\hat{r}^*=30$, respectively. All other parameters remain unchanged (note this does not affect the baseline which regardless of the rank performs only rank-one SVD computations). The complete set of results is given in Figures \ref{fig:r5}, \ref{fig:r10}, \ref{fig:r20}, \ref{fig:r15}, \ref{fig:r30}. We can see that w.r.t. all measures, both of our variants clearly outperform the baseline. 

\begin{figure}[H]
\centering
\begin{multicols}{3}
    \includegraphics[width=\linewidth]{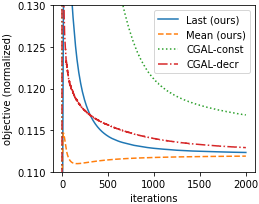}\par
    \includegraphics[width=\linewidth]{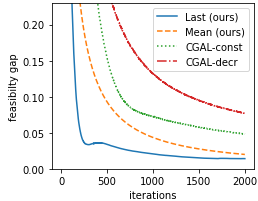}\par  
          \includegraphics[width=\linewidth]{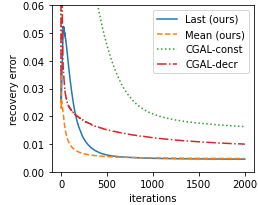}\par
\end{multicols}
\begin{multicols}{3}
    \includegraphics[width=\linewidth]{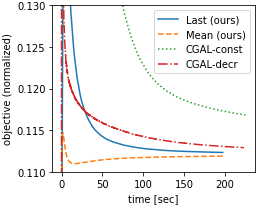}\par
    \includegraphics[width=\linewidth]{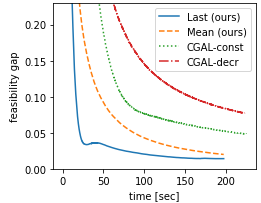}\par
    \includegraphics[width=\linewidth]{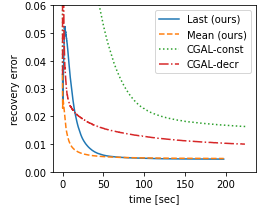}\par
\end{multicols}
\caption{Results for low-rank and sparse covariance estimation for $r=\hat{r}^*= 5$.}
\label{fig:r5}
\end{figure}

\begin{figure}[H]
\centering
\begin{multicols}{3}
    \includegraphics[width=\linewidth]{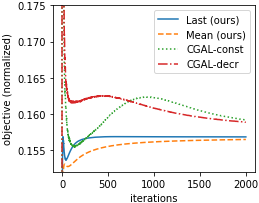}\par
    \includegraphics[width=\linewidth]{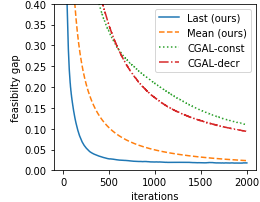}\par        
    \includegraphics[width=\linewidth]{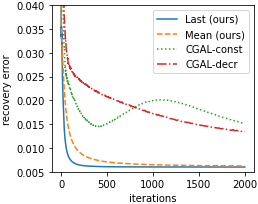}\par
\end{multicols}
\begin{multicols}{3}
    \includegraphics[width=\linewidth]{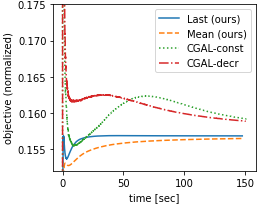}\par
    \includegraphics[width=\linewidth]{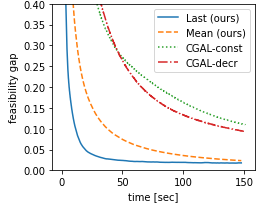}\par
    \includegraphics[width=\linewidth]{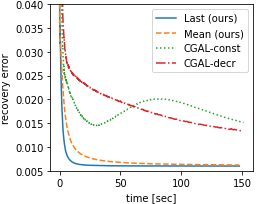}\par
\end{multicols}
\caption{Results for low-rank and sparse covariance estimation for $r=\hat{r}^*= 10$.}
\label{fig:r10}
\end{figure}

\begin{figure}[H]
\centering
\begin{multicols}{3}
    \includegraphics[width=\linewidth]{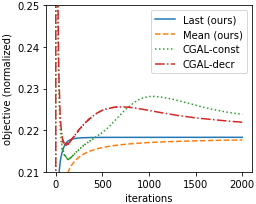}\par
    \includegraphics[width=\linewidth]{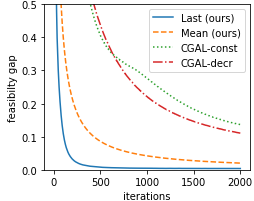}\par        
    \includegraphics[width=\linewidth]{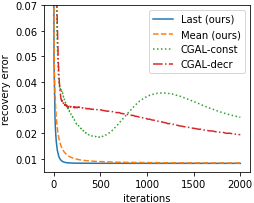}\par
\end{multicols}
\begin{multicols}{3}
    \includegraphics[width=\linewidth]{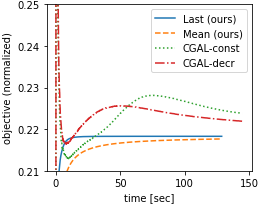}\par
    \includegraphics[width=\linewidth]{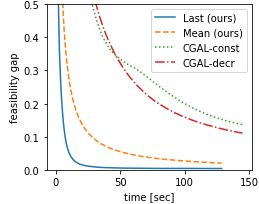}\par
    \includegraphics[width=\linewidth]{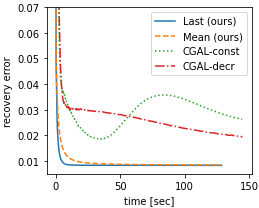}\par
\end{multicols}
\caption{Results for low-rank and sparse covariance estimation  for $r=\hat{r}^*= 20$.}
\label{fig:r20}
\end{figure}

\begin{figure}[H]
\centering
\begin{multicols}{3}
    \includegraphics[width=\linewidth]{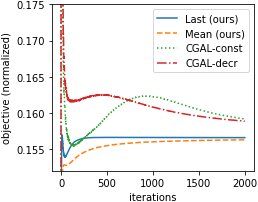}\par
    \includegraphics[width=\linewidth]{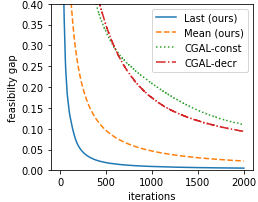}\par        
    \includegraphics[width=\linewidth]{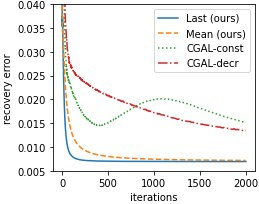}\par
\end{multicols}
\begin{multicols}{3}
    \includegraphics[width=\linewidth]{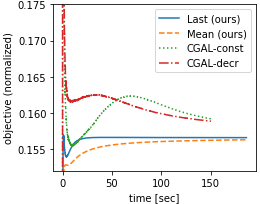}\par
    \includegraphics[width=\linewidth]{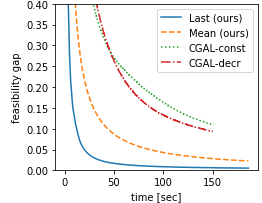}\par
    \includegraphics[width=\linewidth]{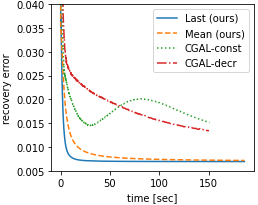}\par
\end{multicols}
\caption{Results for low-rank and sparse covariance estimation for $r= 10$ and  $\hat{r}^*=15$.}
\label{fig:r15}
\end{figure}

\begin{figure}[H]
\centering
\begin{multicols}{3}
    \includegraphics[width=\linewidth]{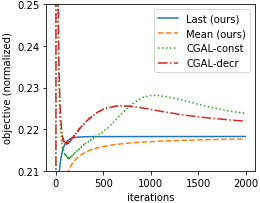}\par
    \includegraphics[width=\linewidth]{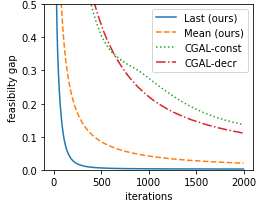}\par        
    \includegraphics[width=\linewidth]{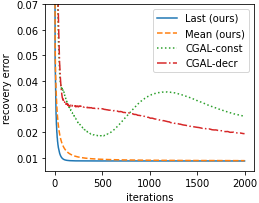}\par
\end{multicols}
\begin{multicols}{3}
    \includegraphics[width=\linewidth]{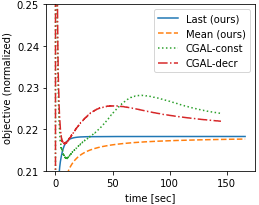}\par
    \includegraphics[width=\linewidth]{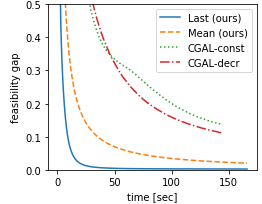}\par
    \includegraphics[width=\linewidth]{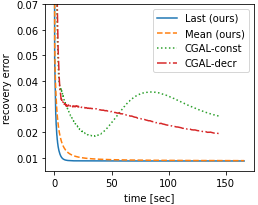}\par
\end{multicols}
\caption{Results for low-rank and sparse covariance estimation for $r= 20$ and  $\hat{r}^*=30$.}
\label{fig:r30}
\end{figure}

\subsection{Max Cut}
We now consider the following well-known semidefinite relaxation for the Max Cut problem:
\begin{equation}\tag{MC}\label{pblm:MC}
        \min_{\Ss}-\textrm{tr}(\CC\Ss) ~~ \textrm{s.t.}~~\textrm{ } \Ss\succeq0, \textrm{tr}(\Ss)=d,\textrm{diag}(\Ss)=\textbf{1},
\end{equation}
where $\CC$ is the Laplacian matrix of a combinatorial graph, $\textrm{diag}(\Ss)$ is the vector of elements on the main diagonal of the matrix $\Ss\in\mathbb{S}^d_+$, and \textbf{1} is a vector of ones of length $d$.

Note that since the objective function is linear, this problem does not necessarily satisfy our primal quadratic gap condition (Definition \ref{def:primal quadratic gap}). Still this does not revoke the applicability of our algorithm for this problem.

\paragraph*{Datasets.} We used the G1, G2 and G3 graphs from the Gset dataset\footnote{Y. Ye. Gset random graphs, found in: https://www.cise.ufl.edu/research/sparse/matrices/Gset/index.html}. These are graphs of size $800\times800$. The ranks of the optimal solutions for the max cut problem in these graphs are conveniently given in \cite{ding2021simplicity}, Table 1 (the ranks of the optimal solutions are $r^*=13$ for G1 and G2, and $r^*=14$ for G3).

\subsubsection{Implementation details}

\paragraph{Our algorithm.}
We set $\Rx$ to be an indicator of the spectrahedron $\s$ with $\tau=d$. We set $\A$ to be the identity operator, and set $\Ry$ to be the indicator function for the set of $d\times d$ matrices with a diagonal of ones. Hence on each iteration, the main computation of our algorithm is a single low-rank SVD. We also set $f(\x)\equiv f(\Ss):=-\textrm{tr}(\CC\Ss)$, for simplicity of computation. As in the CME experiments, we use line search for computing the primal step size on each iteration for this problem as well. In particular, we set the rank estimate for our algorithm, which determines the rank of the SVD computations, as $\hat{r}^* = r^*$.

\paragraph{The baseline algorithm.}
As in the implementation of our algorithm, we set $\A$ to be the identity operator, $\X$ to be the spectrahedron $\s$ with $\tau=d$, and $\Y$ to be the set of $d\times d$ matrices with diagonal of ones. We used the same objective function as in our algorithm. Hence, for the baseline algorithm, the only expensive computation is a rank-one SVD computation on each iteration.

\paragraph{SVD computations and rank overestimation.}
We used the same Python built-in function for thin SVD computations as in the previous experiment. Here we also set the tolerance parameter of the function to 0.01 for all algorithms. For the G1 graph we also considered overestimating the rank of the optimal solution for our algorithm, taking the SVD rank to be $\hat{r}^*=20$, where we know $r^*=13$ is the rank of the optimal solution.

\paragraph{Manual tuning.}
Since this problem does not satisfy the PQG property, the choice of $\eta$ becomes heuristic, which adds another free parameter to our algorithm.
Here we chose $\mu=0.2$ and $\eta=0.2$ arbitrarily, while the choice of $\rho=1$ for both variants of our algorithm, and $\rho_0=1$ for both variants of the baseline was done the same way as in the previous experiment, based on the performance over the graph G1 , where $\hat{r}^*=r^*=13$ (Figure \ref{fig:g1r13}). Here, we used the same parameters for all of four datasets (Figures \ref{fig:g1r13},\ref{fig:g2r13},\ref{fig:g3r14},\ref{fig:g1r20}).

\paragraph{Initialization.}
For both algorithms we initialized $\Ss$ to be the identity matrix, and the dual variable to be zero. In addition, for our algorithm, we initialized the additional variable $\y$ to be the identity matrix as well.

\subsubsection{Results}

We ran all algorithms for 2000 iterations. For each of the datasets we repeated the experiment 10 times and averaged the runtimes  for more reliable measurements. Here also, we initially set the SVD rank parameter for our algorithm to be exactly $r$, i.e., we set $\hat{r}^*=r$, but we also redid the experiment for G1 when overestimating the rank of the optimal solution for our algorithm, taking $\hat{r}^* = 20$ (while $r^*=13$). We plot the objective value and the feasibility gap measured by $\|\textrm{diag}(\Ss)-\mathbf{1}\|$ both w.r.t. number of iterations and runtime. The complete set of results is given in Figures \ref{fig:g1r13}, \ref{fig:g2r13}, \ref{fig:g3r14}, \ref{fig:g1r20}. We can see that when $\hat{r}^* = r$, both of our variants are faster than the baselines, for all of the datasets. When $\hat{r}^* > r$, while our "Mean" variant is slightly surpassed by the baseline w.r.t. the objective,  it converges faster to a feasible solution. Our "Last" variant performs better than the baseline w.r.t. both measures. 

\begin{figure}[H]
\centering
\begin{multicols}{2}
    \includegraphics[width=.65\linewidth]{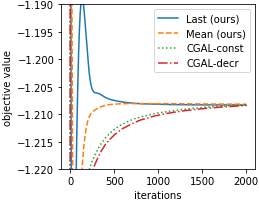}\par
    \includegraphics[width=.65\linewidth]{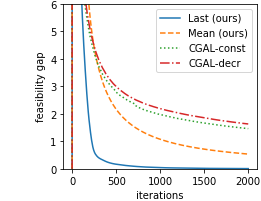}\par    
\end{multicols}
\begin{multicols}{2}
    \includegraphics[width=.65\linewidth]{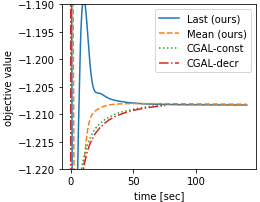}\par
    \includegraphics[width=.65\linewidth]{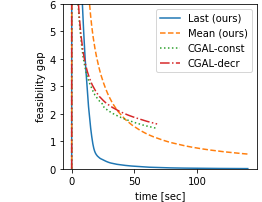}\par
\end{multicols}
\caption{Results for Max Cut with graph G1 ($r^*=\hat{r}^*=13$).}
\label{fig:g1r13}
\end{figure}

\begin{figure}[H]
\centering
\begin{multicols}{2}
    \includegraphics[width=.65\linewidth]{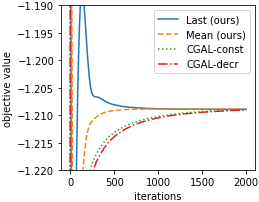}\par
    \includegraphics[width=.65\linewidth]{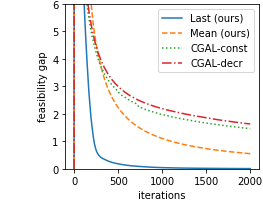}\par
\end{multicols}
\begin{multicols}{2}
    \includegraphics[width=.65\linewidth]{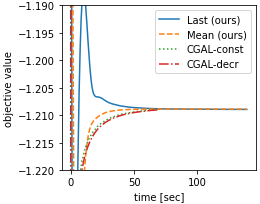}\par
    \includegraphics[width=.65\linewidth]{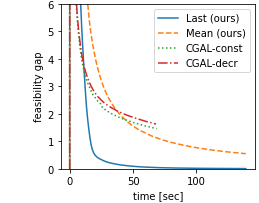}\par
\end{multicols}
\caption{Results for Max Cut with graph G2 ($r^*  = \hat{r}^*=13$).}
\label{fig:g2r13}
\end{figure}

\begin{figure}[H]
\centering
\begin{multicols}{2}
    \includegraphics[width=.65\linewidth]{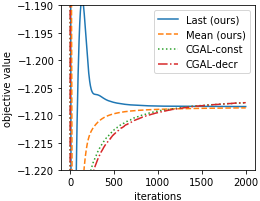}\par
    \includegraphics[width=.65\linewidth]{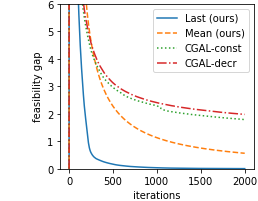}\par    
\end{multicols}
\begin{multicols}{2}
    \includegraphics[width=.65\linewidth]{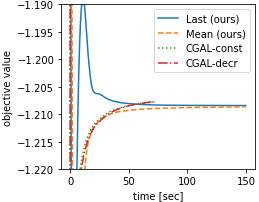}\par
    \includegraphics[width=.65\linewidth]{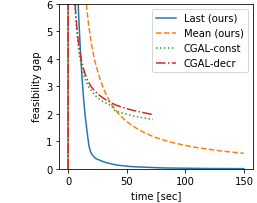}\par
\end{multicols}
\caption{Results for Max Cut with graph G3 ($r^* = \hat{r}^*=14$).}
\label{fig:g3r14}
\end{figure}

\begin{figure}[H]
\centering
\begin{multicols}{2}
    \includegraphics[width=.65\linewidth]{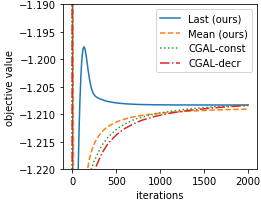}\par
    \includegraphics[width=.65\linewidth]{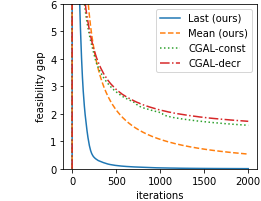}\par
\end{multicols}
\begin{multicols}{2}
    \includegraphics[width=.65\linewidth]{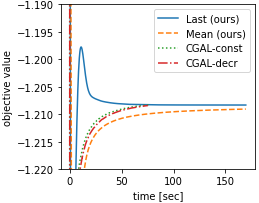}\par    
    \includegraphics[width=.65\linewidth]{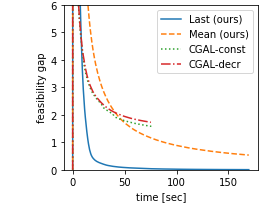}\par
\end{multicols}
\caption{Results for Max Cut with graph G1 ($r^*=13$) for  $\hat{r}^*=20$.}
\label{fig:g1r20}
\end{figure}

\section*{Acknowledgements}
We thank Atara Kaplan for important discussions  regarding  problems with the convergence proofs in \cite{gidel2018frank} (discussed in Section \ref{sec:issue}).

Dan Garber is supported by the ISRAEL SCIENCE FOUNDATION (grant No. 2267/22).

\bibliographystyle{plain}
\bibliography{bibliography}

\begin{thebibliography}{10}

\bibitem{allen2017linear}
Zeyuan Allen-Zhu, Elad Hazan, Wei Hu, and Yuanzhi Li.
\newblock Linear convergence of a {F}rank-{W}olfe type algorithm over
  trace-norm balls.
\newblock {\em Advances in Neural Information Processing Systems}, 30, 2017.

\bibitem{andrews1991heteroskedasticity}
Donald~WK Andrews.
\newblock Heteroskedasticity and autocorrelation consistent covariance matrix
  estimation.
\newblock {\em Econometrica: Journal of the Econometric Society}, pages
  817--858, 1991.

\bibitem{beck2017linearly}
Amir Beck and Shimrit Shtern.
\newblock Linearly convergent away-step conditional gradient for non-strongly
  convex functions.
\newblock {\em Mathematical Programming}, 164(1):1--27, 2017.

\bibitem{buades2005non}
Antoni Buades, Bartomeu Coll, and J-M Morel.
\newblock A non-local algorithm for image denoising.
\newblock In {\em 2005 IEEE computer society conference on computer vision and
  pattern recognition (CVPR'05)}, volume~2, pages 60--65. Ieee, 2005.

\bibitem{buades2005review}
Antoni Buades, Bartomeu Coll, and Jean-Michel Morel.
\newblock A review of image denoising algorithms, with a new one.
\newblock {\em Multiscale modeling \& simulation}, 4(2):490--530, 2005.

\bibitem{candes2011robust}
Emmanuel~J Cand{\`e}s, Xiaodong Li, Yi~Ma, and John Wright.
\newblock Robust principal component analysis?
\newblock {\em Journal of the ACM (JACM)}, 58(3):1--37, 2011.

\bibitem{candes2009exact}
Emmanuel~J Cand{\`e}s and Benjamin Recht.
\newblock Exact matrix completion via convex optimization.
\newblock {\em Foundations of Computational mathematics}, 9(6):717--772, 2009.

\bibitem{candes2012exact}
Emmanuel~J Candes and Benjamin Recht.
\newblock Exact matrix completion via convex optimization.
\newblock {\em Communications of the ACM}, 55(6):111--119, 2012.

\bibitem{chambolle2016ergodic}
Antonin Chambolle and Thomas Pock.
\newblock On the ergodic convergence rates of a first-order primal--dual
  algorithm.
\newblock {\em Mathematical Programming}, 159(1):253--287, 2016.

\bibitem{chenna2003multiple}
Ramu Chenna, Hideaki Sugawara, Tadashi Koike, Rodrigo Lopez, Toby~J Gibson,
  Desmond~G Higgins, and Julie~D Thompson.
\newblock Multiple sequence alignment with the clustal series of programs.
\newblock {\em Nucleic acids research}, 31(13):3497--3500, 2003.

\bibitem{corpet1988multiple}
Florence Corpet.
\newblock Multiple sequence alignment with hierarchical clustering.
\newblock {\em Nucleic acids research}, 16(22):10881--10890, 1988.

\bibitem{de2022constrained}
Alberto De~Marchi.
\newblock Constrained structured optimization and augmented {L}agrangian
  proximal methods.
\newblock {\em arXiv preprint arXiv:2203.05276}, 2022.

\bibitem{dhingra2018proximal}
Neil~K Dhingra, Sei~Zhen Khong, and Mihailo~R Jovanovi{\'c}.
\newblock The proximal augmented {L}agrangian method for nonsmooth composite
  optimization.
\newblock {\em IEEE Transactions on Automatic Control}, 64(7):2861--2868, 2018.

\bibitem{ding2021simplicity}
Lijun Ding and Madeleine Udell.
\newblock On the simplicity and conditioning of low rank semidefinite programs.
\newblock {\em SIAM Journal on Optimization}, 31(4):2614--2637, 2021.

\bibitem{driscoll1998consistent}
John~C Driscoll and Aart~C Kraay.
\newblock Consistent covariance matrix estimation with spatially dependent
  panel data.
\newblock {\em Review of economics and statistics}, 80(4):549--560, 1998.

\bibitem{gandy2011tensor}
Silvia Gandy, Benjamin Recht, and Isao Yamada.
\newblock Tensor completion and low-n-rank tensor recovery via convex
  optimization.
\newblock {\em Inverse problems}, 27(2):025010, 2011.

\bibitem{garber2016faster}
Dan Garber.
\newblock Faster projection-free convex optimization over the spectrahedron.
\newblock {\em Advances in Neural Information Processing Systems}, 29, 2016.

\bibitem{garber2019logarithmic}
Dan Garber.
\newblock Logarithmic regret for online gradient descent beyond strong
  convexity.
\newblock In {\em The 22nd International Conference on Artificial Intelligence
  and Statistics}, pages 295--303. PMLR, 2019.

\bibitem{garber2016linearly}
Dan Garber and Elad Hazan.
\newblock A linearly convergent variant of the conditional gradient algorithm
  under strong convexity, with applications to online and stochastic
  optimization.
\newblock {\em SIAM Journal on Optimization}, 26(3):1493--1528, 2016.

\bibitem{garber2019fast}
Dan Garber and Atara Kaplan.
\newblock Fast stochastic algorithms for low-rank and nonsmooth matrix
  problems.
\newblock In {\em The 22nd International Conference on Artificial Intelligence
  and Statistics}, pages 286--294. PMLR, 2019.

\bibitem{garber2018fast}
Dan Garber, Atara Kaplan, and Shoham Sabach.
\newblock Improved complexities of conditional gradient-type methods with
  applications to robust matrix recovery problems.
\newblock {\em Mathematical Programming}, 186(1):185--208, 2021.

\bibitem{gidel2018frank}
Gauthier Gidel, Fabian Pedregosa, and Simon Lacoste-Julien.
\newblock Frank-{W}olfe splitting via augmented {L}agrangian method.
\newblock In {\em International Conference on Artificial Intelligence and
  Statistics}, pages 1456--1465. PMLR, 2018.

\bibitem{gramfort2013identifying}
Alexandre Gramfort, Bertrand Thirion, and Ga{\"e}l Varoquaux.
\newblock Identifying predictive regions from fmri with tv-l1 prior.
\newblock In {\em 2013 International Workshop on Pattern Recognition in
  Neuroimaging}, pages 17--20. IEEE, 2013.

\bibitem{HazanK12}
Elad Hazan and Satyen Kale.
\newblock Projection-free online learning.
\newblock In {\em Proceedings of the 29th International Conference on Machine
  Learning, {ICML} 2012, Edinburgh, Scotland, UK, June 26 - July 1, 2012}.
  icml.cc / Omnipress, 2012.

\bibitem{hestenes1969multiplier}
Magnus~R Hestenes.
\newblock Multiplier and gradient methods.
\newblock {\em Journal of optimization theory and applications}, 4(5):303--320,
  1969.

\bibitem{jaggi2013revisiting}
Martin Jaggi.
\newblock Revisiting {F}rank-{W}olfe: Projection-free sparse convex
  optimization.
\newblock In {\em International Conference on Machine Learning}, pages
  427--435. PMLR, 2013.

\bibitem{johan1990tensor}
Hastad Johan.
\newblock Tensor rank is np-complete.
\newblock {\em Journal of Algorithms}, 4(11):644--654, 1990.

\bibitem{katoh2008recent}
Kazutaka Katoh and Hiroyuki Toh.
\newblock Recent developments in the mafft multiple sequence alignment program.
\newblock {\em Briefings in bioinformatics}, 9(4):286--298, 2008.

\bibitem{kolda2009tensor}
Tamara~G Kolda and Brett~W Bader.
\newblock Tensor decompositions and applications.
\newblock {\em SIAM review}, 51(3):455--500, 2009.

\bibitem{lacoste2015global}
Simon Lacoste-Julien and Martin Jaggi.
\newblock On the global linear convergence of {F}rank-{W}olfe optimization
  variants.
\newblock {\em Advances in neural information processing systems}, 28, 2015.

\bibitem{lancaster2000automated}
Jack~L Lancaster, Marty~G Woldorff, Lawrence~M Parsons, Mario Liotti,
  Catarina~S Freitas, Lacy Rainey, Peter~V Kochunov, Dan Nickerson, Shawn~A
  Mikiten, and Peter~T Fox.
\newblock Automated talairach atlas labels for functional brain mapping.
\newblock {\em Human brain mapping}, 10(3):120--131, 2000.

\bibitem{liben2003link}
David Liben-Nowell and Jon Kleinberg.
\newblock The link prediction problem for social networks.
\newblock In {\em Proceedings of the twelfth international conference on
  Information and knowledge management}, pages 556--559, 2003.

\bibitem{liu2019nonergodic}
Ya-Feng Liu, Xin Liu, and Shiqian Ma.
\newblock On the nonergodic convergence rate of an inexact augmented
  {L}agrangian framework for composite convex programming.
\newblock {\em Mathematics of Operations Research}, 44(2):632--650, 2019.

\bibitem{lu2011link}
Linyuan L{\"u} and Tao Zhou.
\newblock Link prediction in complex networks: A survey.
\newblock {\em Physica A: statistical mechanics and its applications},
  390(6):1150--1170, 2011.

\bibitem{mizoguchi1960kj}
Toshiyuki Mizoguchi et~al.
\newblock Kj arrow, l. hurwicz and h. uzawa, studies in linear and non-linear
  programming.
\newblock {\em Economic Review}, 11(3):349--351, 1960.

\bibitem{ogawa1992intrinsic}
Seiji Ogawa, David~W Tank, Ravi Menon, Jutta~M Ellermann, Seong~G Kim, Helmut
  Merkle, and Kamil Ugurbil.
\newblock Intrinsic signal changes accompanying sensory stimulation: functional
  brain mapping with magnetic resonance imaging.
\newblock {\em Proceedings of the National Academy of Sciences},
  89(13):5951--5955, 1992.

\bibitem{powell1969method}
Michael~JD Powell.
\newblock A method for nonlinear constraints in minimization problems.
\newblock {\em Optimization}, pages 283--298, 1969.

\bibitem{richard2012estimation}
Emile Richard, Pierre-Andr{\'e} Savalle, and Nicolas Vayatis.
\newblock Estimation of simultaneously sparse and low rank matrices.
\newblock {\em arXiv preprint arXiv:1206.6474}, 2012.

\bibitem{rockafellar1976augmented}
R~Tyrrell Rockafellar.
\newblock Augmented {L}agrangians and applications of the proximal point
  algorithm in convex programming.
\newblock {\em Mathematics of operations research}, 1(2):97--116, 1976.

\bibitem{sabach2019lagrangian}
Shoham Sabach and Marc Teboulle.
\newblock Lagrangian methods for composite optimization.
\newblock {\em Handbook of Numerical Analysis}, 20:401--436, 2019.

\bibitem{silveti2020generalized}
Antonio Silveti-Falls, Cesare Molinari, and Jalal Fadili.
\newblock Generalized conditional gradient with augmented lagrangian for
  composite minimization.
\newblock {\em SIAM Journal on Optimization}, 30(4):2687--2725, 2020.

\bibitem{yen2016convex}
Ian En-Hsu Yen, Xin Lin, Jiong Zhang, Pradeep Ravikumar, and Inderjit Dhillon.
\newblock A convex atomic-norm approach to multiple sequence alignment and
  motif discovery.
\newblock In {\em International Conference on Machine Learning}, pages
  2272--2280. PMLR, 2016.

\bibitem{yurtsever2019conditional}
Alp Yurtsever, Olivier Fercoq, and Volkan Cevher.
\newblock A conditional-gradient-based augmented {L}agrangian framework.
\newblock In {\em International Conference on Machine Learning}, pages
  7272--7281. PMLR, 2019.

\bibitem{zhang2017link}
Jiawei Zhang, Jianhui Chen, Shi Zhi, Yi~Chang, S~Yu Philip, and Jiawei Han.
\newblock Link prediction across aligned networks with sparse and low rank
  matrix estimation.
\newblock In {\em 2017 IEEE 33rd International Conference on Data Engineering
  (ICDE)}, pages 971--982. IEEE, 2017.

\bibitem{zhang2017beyond}
Kai Zhang, Wangmeng Zuo, Yunjin Chen, Deyu Meng, and Lei Zhang.
\newblock Beyond a {G}aussian denoiser: Residual learning of deep cnn for image
  denoising.
\newblock {\em IEEE transactions on image processing}, 26(7):3142--3155, 2017.

\bibitem{zhang2020image}
Yongqin Zhang, Ruiwen Kang, Xianlin Peng, Jun Wang, Jihua Zhu, Jinye Peng, and
  Hangfan Liu.
\newblock Image denoising via structure-constrained low-rank approximation.
\newblock {\em Neural Computing and Applications}, 32(16):12575--12590, 2020.

\end{thebibliography}

\newpage
\appendix

\section{Issue with the Convergence Proofs in \cite{gidel2018frank}}\label{sec:issue}
In Section C of the paper (see pages 13-14), given a compact set $\X$ (formally defined in Subsection 2.2 on page 3), the authors define the following functions (equation numbers are as in the original paper):
\begin{align*}
    &\mathcal{L}:=f(\x)+\one(\x)+\langle \y,\MM\x\rangle+\frac{\lambda}{2}\|\MM\x\|^2,\tag{45}\\
    &d(\y):=\min_{\x\in\X}\mathcal{L}(\x,\y),\tag{46}\\
    &d^*=\max_\y d(\y).
\end{align*}
In addition, the authors denote the smoothness parameter of $f(\x)+\frac{\lambda}{2}\|\MM\x\|^2$ by $\Ll$ (assuming $f$ is smooth), the set of optimal dual solutions is denoted by $\Y^*$, the minimal distance between $\y$ and a point in $\Y^*$is denoted by dist$(\y,\Y^*)$, and the diameter of $\X$ is denoted by $D$ (defined as the maximum norm between two points in $\X$).

Through these notations, the authors state the following theorem, which is crucial for their convergence analysis.
\begin{theorem}[Theorem 1 in \cite{gidel2018frank}]
Under Slater's condition, there exists $\alpha>0$, such that for any dual variable $\y$, the following holds:
\begin{equation}\tag{92}
    d^*-d(\y)\geq\frac{1}{2\Ll D^2}\min\{\alpha^2\emph{dist}^2(\y,\Y^*),\alpha\Ll D^2\emph{dist}(\y,\Y^*)\}.
\end{equation}
\end{theorem}
In order to prove this result, they define the function
\begin{equation}\tag{48}
    \fl(\x):=f(\x)+\one(\x)+\frac{\lambda}{2}\|\MM\x\|^2
\end{equation}
and at the beginning of the proof of Lemma 1 (on page 15), for any $\x\in\X$ and any given $\uu\in\R{m}$, the authors also define the function:
\begin{equation*}\tag{58}
    g_{\x}(\uu):=\fl(\uu+\x)-\fl(\x)-\langle\uu,\n\rangle,\quad\n\in\partial\fl(\x).
\end{equation*}
From the definition of $\fl$ we can equivalently write it as follows
\begin{equation*}
    (f+\frac{1}{2}\|\MM\cdot\|^2)(\uu+\x)-(f+\frac{1}{2}\|\MM\cdot\|^2)(\x)-\langle\uu,\nabla (f+\frac{1}{2}\|\MM\cdot\|^2)(\x)\rangle\\+\one(\uu+\x)-\one(\x)-\langle\uu,\n'\rangle,
\end{equation*}
where $n' \in \partial\one(\x) =N_c^{\X}(\x) :=\{\n'\in\R{m}|\langle \n',\x-\x'\rangle\geq0,\forall\x'\in\X\}$ is the normal cone of $\X$.
\\\\
Now, right after the definition of $g_{\x}(\uu)$ in Eq. (58), the authors define the following function
\begin{equation*}
    h_{\x}(\uu):=\frac{\Ll}{2}\|\uu\|^2+\one(\uu+\x),
\end{equation*}
and use the Descent  Lemma\footnote{Recall that in this case, the Descent Lemma reads as follows\\ $(f+\frac{1}{2}\|\MM\cdot\|^2)(\uu+\x)-(f+\frac{1}{2}\|\MM\cdot\|^2)(\x)-\langle\uu,\nabla (f+\frac{1}{2}\|\MM\cdot\|^2)(\x)\rangle\leq\frac{L_{\lambda}}{2}\|\uu\|^2$.} to claim that 
\begin{align}\label{eq:falseEq}
\forall\uu\in\R{m}: \quad g_{\x}(\uu)\leq h_{\x}(\uu).
\end{align}
However, we observe that using the Descent Lemma and the fact that $\x\in\X$, we obtain that
\begin{equation*}
    g_{\x}(\uu)\leq\frac{L_{\lambda}}{2}\|\uu\|^2+\one(\uu+\x)\color{red}-\langle\uu,\n'\rangle\color{black}
    = h_{\x}(\uu)\color{red}-\langle\uu,\n'\rangle\color{black}.
\end{equation*}
Now, if $\uu+\x\in\X$, then the rightmost term highlighted in red (including the minus sign) is non-negative by the definition of the normal cone (taking $\x'=\uu+\x$), and thus can't be omitted, and so the inequality \eqref{eq:falseEq} cannot be deduced.

Adding the missing term throughout the rest of the proof, we eventually get a more complicated inequality instead of (75) in the original paper. Thus, the value of the constant $\alpha$ in the theorem above, for which the positivity is proven (see (76) in \cite{gidel2018frank}), is irrelevant, and the fixed expression for $\alpha$ for which we can obtain the error bound property (see (92)), becomes even more complicated. It is then unclear that $\alpha$ can be proven to be strictly positive, which is required for the convergence results in \cite{gidel2018frank}.

\section{Smoothness and Primal Quadratic Gap Bounds for the Augmented Lagrangian}\label{sec:properties proofs}
\subsection{Proof of Lemma \ref{lemma:loose upper bound for K}(smoothness of $S$ w.r.t. the primal variable)}
\begin{proof}
    Since $f(\x)$ is $\beta$-smooth we obviously have that $S(\cdot,\w)$, for any fixed $\w \in \E_{2}$,  is $(\beta+\rho\|\K\|^2)$-smooth. Moreover, recalling the definition of the  spectral norm, we have    
    \begin{equation*}
        \begin{split}
            \|\K\|:=\max_{\|\q\|^2=1}\|\K\q\|&=\max_{\|\x\|^2+\|\y\|^2=1}\|\A \x-\y\|\\&\leq\max_{\|\x\|^2+\|\y\|^2=1}\|\A\x\|+\|\y\| \\&\leq\max_{\|\x\|^2=1}\|\A\x\| + \max_{\|\y\|^2=1}\|\y\|^2\\&=\|\A\|+1.
        \end{split}
    \end{equation*}
    Hence, the desired result follows.
\end{proof}

\subsection{Proof of Theorem \ref{thm:first case} (primal quadratic gap when $f$ is strongly convex)}

\begin{proof}
    Let us define $\phi_\w(\q):=S(\q,\w)$ for any $\q\in\E$ and any $\w\in\E_2$. We start the proof by presenting two properties of $\phi_\w$, that hold regardless of the strong convexity of $f$.
    \\
    First, from the differentiability of $f$, for any $\q=\cv{\x}{\y}\in\E$, we have that    
    \begin{equation}\label{eq:derivative}
        \nabla\phi_\w(\q)=\cv{\nabla f(\x)}{\zero}+\K^{\top}\w+\rho \K^{\top}\K\q.
    \end{equation}
    Second, from the definition of $\phi_\w$, we get, for any $\q_1=\cv{\x_1}{\y_1},\q_2=\cv{\x_2}{\y_2}\in\E$, that    
    \begin{align}
            \phi_\w(\q_2)-\phi_\w(\q_1)=&f(\x_2)-f(\x_1)+\langle\w,\K(\q_2-\q_1)\rangle+\frac{\rho}{2}\|\K\q_2\|^2-\frac{\rho}{2}\|\K\q_1\|^2 \nonumber\\
            =&f(\x_2)-f(\x_1)+\langle \K^{\top}\w,\q_2-\q_1\rangle+\rho\langle \K\q_1,\K(\q_2-\q_1)\rangle +\frac{\rho}{2}\|\K\q_2-\K\q_1\|^2 \nonumber\\
            =&f(\x_2)-f(\x_1)+\langle \K^{\top}\w+\rho\K^{\top}\K\q_1,\q_2-\q_1\rangle+\frac{\rho}{2}\|\K\q_2-\K\q_1\|^2. \label{eq:proof body}
    \end{align}
    Now, from the $\alpha$-strong convexity of $f(\x)$, we have, for any $\x_1,\x_2\in\E_1$, that    
    \begin{equation}\label{eq:scX}
        f(\x_2)-f(\x_1)\geq\langle \nabla f(\x_1),\x_2-\x_1\rangle+\frac{\alpha}{2}\|\x_2-\x_1\|^2.
    \end{equation}
    From \eqref{eq:derivative}, \eqref{eq:proof body} and \eqref{eq:scX} we further have    
    \begin{align}
            \phi_\w(\q_2)-\phi_\w(\q_1) &\geq\langle \nabla f(\x_1),\x_2-\x_1\rangle+\frac{\alpha}{2}\|\x_2-\x_1\|^2 +\frac{\rho}{2}\|\K\q_2-\K\q_1\|^2\\&+\langle \K^{\top}\w+\rho\K^{\top}\K\q_1,\q_2-\q_1\rangle. \nonumber\\
            &=\langle\nabla\phi_\w(\q_1),\q_2-\q_1\rangle+\frac{\alpha}{2}\|\x_2-\x_1\|^2+\frac{\rho}{2}\|\A\x_2-\A\x_1+\y_2-\y_1\|^2, \label{eq:body cont}
    \end{align}
    where the last equality also follows from the definition of $\K$.
    \\\\
    Now, since for every $a,b\in\R{}$, and $s>0$ it holds that
    $$(a+b)^2\geq(1-s)a^2+\frac{s-1}{s}b^2,$$
    we get that for every $s\geq1$    
    \begin{align}
            \|\A\x_2-\A\x_1+\y_2-\y_1\|^2&\geq(\|\A\x_2-\A\x_1\|-\|\y_2-\y_1\|)^2 \nonumber\\
            &\geq(1-s)\|\A\x_2-\A\x_1\|^2+\frac{s-1}{s}\|\y_2-\y_1\|^2 \nonumber\\
            &\geq\|\A\|^2(1-s)\|\x_2-\x_1\|^2+\frac{s-1}{s}\|\y_2-\y_1\|^2,\label{eq:s hack}
    \end{align}
    where the first inequality follows from the well-known Cauchy-Schwarz inequality.\\\\
    Combining \eqref{eq:body cont} and \eqref{eq:s hack}, and denoting $\delta(s):=\min\{\alpha+\rho\|\A\|^2(1-s),\rho(s-1)/s\},$ we get for every $s\geq1$, that    
    \begin{align*}
            \phi_\w(\q_2)-\phi_\w(\q_1)& \geq \langle \nabla\phi_\w(\q_1),\q_2-\q_1\rangle+\frac{1}{2}\left(\alpha+\rho\|\A\|^2(1-s)\right)\|\x_2-\x_1\|^2\\&+\frac{\rho(s-1)}{2s}\|\y_2-\y_1\|^2\\
            & \geq \langle \nabla\phi_\w(\q_1),\q_2-\q_1\rangle+\frac{\delta(s)}{2}(\|\x_2-\x_1\|^2+\|\y_2-\y_1\|^2)\\
            &= \langle \nabla\phi_\w(\q_1),\q_2-\q_1\rangle+\frac{\delta(s)}{2}\|\q_2-\q_1\|^2.
    \end{align*}
    The result now follows, since taking $\tilde{s}=1+\frac{\alpha}{2\rho\|\A\|^2}>1$, we have  
    \begin{equation*}\label{eq:aLp}
        \begin{split}
            \delta(\tilde{s}) &=\as =\min\left\{\frac{\alpha}{2},\frac{\alpha\rho}{\alpha+2\rho\|\A\|^2}\right\}>0,
        \end{split}
    \end{equation*}
    as required.
\end{proof}

\subsection{Proof of Theorem \ref{thm:second case} (primal quadratic gap when $\Rq$ is an indicator for a polytope)}\label{sec:pqg for polytopes}
The strong duality of a primal dual saddle point problem ensures that for every primal optimal solution there exists a dual optimal solution (and vice versa) such that the pair forms a saddle point. In the case of Problem \eqref{pblm:reformulation}, we can observe the following crucial property, which proves that the set of saddle points is the entire set $\Pp\times\D$, meaning that if strong duality holds, every couple consisting of an optimal solution of Problem \eqref{pblm:problem} and an optimal solution of the dual problem associated with \eqref{pblm:problem}, is a saddle point of $\Lp$.
\begin{proposition}\label{prp:saddle points set}
    Let $(\q_1^*,\w_1^*)$ and $(\q_2^*,\w_2^*)$ be two saddle points of $\Lp$. Then, $(\q_1^*,\w_2^*)$ is also a saddle point of $\Lp$.
\end{proposition}
\begin{proof}
    Since $\q_1^*$ and $\q_2^*$ are optimal solutions of Problem \eqref{pblm:problem}, they are in particular feasible solutions of Problem \eqref{pblm:problem} and thus satisfy $\K\q_1^*=\K\q_2^*=\zero$. Hence, recalling the definition of $\Lp$ we have that $\Lp(\q_1^*,\w_1^*)=\Lp(\q_1^*,\w_2^*)$ and that $\Lp(\q_2^*,\w_1^*)=\Lp(\q_2^*,\w_2^*)$.\\\\
    Therefore, applying the saddle point inequality \eqref{eq:saddle point} with $(\q_1^*,\w_1^*)$, we have for all $\w\in\E_2$, that    
    \begin{equation}\label{eq:left saddle}
        \Lp(\q_1^*,\w)\leq\Lp(\q_1^*,\w_1^*)=\Lp(\q_1^*,\w_2^*).
    \end{equation}
    In addition, applying \eqref{eq:saddle point} with $(\q_2^*,\w_2^*)$, we have for all $\q\in\E$, that    
    \begin{equation}\label{eq:right saddle}
        \Lp(\q,\w_2^*)\geq\Lp(\q_2^*,\w_2^*)=\Lp(\q_2^*,\w_1^*)\geq\Lp(\q_1^*,\w_1^*)=\Lp(\q_1^*,\w_2^*).
    \end{equation}
    Note that the second inequality is another application of \eqref{eq:saddle point} with the saddle point $(\q_1^*,\w_1^*)$.\\\\
    By \eqref{eq:left saddle} and \eqref{eq:right saddle} we get that $(\q_1^*,\w_2^*)$ is a saddle point.
\end{proof}

\begin{lemma}\label{lemma:equality w.r.t. B}
    Let $(\x^*_1,\y^*_1)\equiv\q^*_1$ and $(\x^*_2,\y^*_2)\equiv\q^*_2$ be two optimal solutions of Problem \eqref{pblm:problem}. Then, $\B\x^*_1=\B\x^*_2$.
\end{lemma}
\begin{proof}
    Let $\w^*\in\cal{D}^*$. From the optimality of $\q^*_1,\q^*_2,\w^*$, and by Assumptions \ref{asmp:saddle} and \ref{asmp:slater}, as well as Proposition \ref{prp:saddle points set}, we have that both $(\q_1^*,\w^*)$ and $(\q_2^*,\w^*)$ are saddle points. By the saddle point property given in \eqref{eq:saddle point}, we have    
    \begin{equation}\label{eq:optimality w.r.t. w*}
        \Lp(\q^*_1,\w^*)=\Lp(\q^*_2,\w^*)=\min_{\q\in\E}\Lp(\q,\w^*).
    \end{equation}
    Assume on the contrary that $\B\x^*_1\neq\B\x^*_2$. Then, by the $\alpha_g$-strong convexity of $g$, we have    
    \begin{equation}\label{eq:g}
            g\left(\B\left(\frac{\x^*_1+\x^*_2}{2}\right)\right)\leq\frac{1}{2}(g(\B\x^*_1)+g(\B\x^*_2))-\frac{\alpha_g}{8}\|\B\x^*_1-\B\x^*_2\|^2<\frac{1}{2}(g(\B\x^*_1)+g(\B\x^*_2)).
    \end{equation}
    From the convexity of $\Rx$ and $\Ry$, we have that $\Rq(\q)$ is convex. Hence, the function $$\psi(\q):=\Rq(\q)+\langle \w^*,\K\q\rangle+\frac{\rho}{2}\|\K\q\|^2,$$ is convex in $\q$. Thus, from \eqref{eq:optimality w.r.t. w*} and \eqref{eq:g}, we have
    \begin{align*}
            \Lp\left(\frac{\q^*_1+\q^*_2}{2},\w^*\right) & = g\left(\B\left(\frac{\x^*_1+\x^*_2}{2}\right)\right)+\psi\left(\frac{\q^*_1+\q^*_2}{2}\right) \\
            & <\frac{1}{2}g(\B\x^*_1)+\frac{1}{2}g(\B\x^*_2)+\frac{1}{2}\psi(\q^*_1)+\frac{1}{2}\psi(\q^*_2)\\
            & = \frac{1}{2}\Lp(\q^*_1,\w^*)+\frac{1}{2}\Lp(\q^*_2,\w^*)\\
            &= \min_{\q\in\E}\Lp(\q,\w^*),
    \end{align*}
    which is a contradiction.
\end{proof}
The following Lemma is a known property of polytopes, which is stated without a proof.  For a proof we refer the reader to Lemma 4 in \cite{garber2019logarithmic}.
\begin{lemma}(Hoffman's Lemma)
    Let $\F := \{\q\in\E | \C\q \leq \bb\}$ be a compact and convex polytope and let $\T: \V_1\to\V_2$ be a linear mapping. Given some $\cc \in \V_2$, we define the set $\F(\T, \cc) := \{\q \in \F | \T\q = \cc\}$. If $\F(\T, \cc) \neq \emptyset$, then there exists $\sigma > 0$ such that for any $\q \in \cal{P}$ we have
    $$\emph{dist}(\q,\F(\T, \cc))\leq\sigma\|\T\q - \cc\|^2,$$
    where we define $\emph{dist}(\q,\F(\T, \cc)):=\min_{\z\in\F(\T, \cc)}\|\z-\q\|^2$.
\end{lemma}
\begin{remark}
    The value of $\sigma$ depends on $\C$ and $\T$.
\end{remark}
\begin{lemma}\label{lemma:alternative to sc of f for polytopes}
    There exists a constant $\sigma>0$, such that for any $\q\equiv(\x,\y)\in\F$, we have    
    \begin{equation}
        f(\x^*)-f(\x)\geq\langle\nabla f(\x),\x^*-\x\rangle+\frac{\alpha_g\sigma^{-1}}{2}\|\q^*-\q\|^2-\frac{\alpha_g}{2}\|\K\q\|^2,
    \end{equation}
    where $\q^*\equiv(\x^*,\y^*)\in\Pp$ is the projection of $\q$ onto $\Pp$.
\end{lemma}
\begin{proof}
    From Lemma \ref{lemma:equality w.r.t. B}, it follows that there exists some $\bb^* \in \E_3$, such that for every optimal solution $\x^*$ of Problem \eqref{pblm:problem}, we have that $\B\x^*=\bb^*$.\\\\
    Now, let us denote $\F_{\bb^*}:=\{\q\in\F:\B\x=\bb^*,\K\q=\textbf{0}\}$. We will now show that $\F_{\bb^*}=\Pp$.\\\\
    Let $\q \in \Pp$, then obviously $\K\q = \textbf{0}$ and $\q \in \F$ (a feasible solution) and from Lemma \ref{lemma:equality w.r.t. B} it follows that $\B\x = \bb^*$. Therefore, $\Pp \subseteq \F_{\bb^*}$.\\\\
    On the other hand, any point $\q\in\F_{\bb^*}$ is feasible, as it satisfies $\K\q=\zero$. Moreover, by the choice of $\bb^*$, any $\q\equiv(\x,\y)\in\F_{\bb^*}$ satisfies $\B\x=\bb^*=\B\x^*$, for any optimal primal solution $\q^*\equiv(\x^*,\A\x^*)\in\Pp$. Now, since we also have that $\q\in\F$, the value of the objective of $\q$ satisfies, for any such point $\q^*$ $$g(\B\x)+\Rq(\q)=g(\bb^*)=g(\B\x^*)+\Rq(\q^*),$$ which means that the objective value of any $\q\in\F_{\bb^*}$ is optimal. Therefore, $\q\in\Pp$ and we have that $\F_{\bb^*}\subseteq\Pp$, and thus $\Pp=\F_{\bb^*}$.
    \\\\
    Let $0_{\E_2}:\E_2\to\E_2$ be the zero linear operator $0_{\E_2}(\y)=\zero$. By applying Hoffman's Lemma with $$\Tilde{\B}:=[\B,0_{\E_2}],\quad \T:=\cv{\Tilde{\B}}{\K},\quad \cc:=\cv{\bb^*}{\textbf{0}},$$ there exists a constant $\sigma>0$, such that for any $\q\in\F$, we have (notice that $\tilde{\B}\q = \B\x$)    
    \begin{equation*}
        \begin{split}
            \textrm{dist}(\q,\F_{\bb^*})&=\min_{\z\in\F_{\bb^*}}\|\z-\q\|^2\\&\leq\sigma\|\T\q-\cc\|^2\\&=\sigma\left\|\cv{\Tilde{\B}\q-\bb^*}{\K\q}\right\|^2\\&=\sigma(\|\Tilde{\B}\q-\bb^*\|^2+\|\K\q\|^2)\\&=\sigma(\|\B\x-\bb^*\|^2+\|\K\q\|^2).
        \end{split}
    \end{equation*}
    Thus, by denoting $\q^*:=\argmin_{\z\in\F_{\bb^*}}\|\z-\q\|^2$, we obtain that    
    \begin{equation}\label{eq:hoff rearrange}
        \|\B\x-\bb^*\|^2\geq\sigma^{-1}\|\q^*-\q\|^2-\|\K\q\|^2.
    \end{equation}
    Now, from \eqref{eq:hoff rearrange}, the $\alpha_g$-strong convexity of $g$, and the optimality of $\q^*=\cv{\x^*}{\y^*}$, we have    
    \begin{align*}
            f(\x^*)-f(\x)&=g(\B\x^*)-g(\B\x)\\
            &\geq\langle\x^*-\x,\B^{\top}\nabla g(\B\x)\rangle+\frac{\alpha_g}{2}\|\B\x^*-\B\x\|^2\\
            &=\langle\x^*-\x,\B^{\top}\nabla g(\B\x)\rangle+\frac{\alpha_g}{2}\|\bb^*-\B\x\|^2\\
            &\geq\langle\x^*-\x,\nabla f(\x)\rangle+\frac{\alpha_g\sigma^{-1}}{2}\|\q^*-\q\|^2-\frac{\alpha_g}{2}\|\K\q\|^2.
    \end{align*}
    This completes the desired result since $\Pp=\F_{\bb^*}$. 
\end{proof}
Applying Lemma \ref{lemma:alternative to sc of f for polytopes}, as well as the two properties of $\phi_\w(\q)$ presented in \eqref{eq:derivative} and \eqref{eq:proof body}, we can now prove Theorem \ref{thm:second case}.
\begin{proof}[Proof of Theorem \ref{thm:second case}]
    For any $\q\in\F$, we denote $\q^*=\argmin_{\z\in\Pp}\|\z-\q\|^2$. Recalling the beginning of the proof of Theorem \ref{thm:first case}, by plugging $\q_1=\q\equiv(\x,\y)$ and $\q_2=\q^*\equiv(\x^*,\y^*)$ in \eqref{eq:proof body}, we have (recall that $\K\q^* = \zero$)    
    \begin{equation*}\label{eq:body}
                \phi_\w(\q^*)-\phi_\w(\q) =f(\x^*)-f(\x)+\langle \K^{\top}\w+\rho\K^{\top}\K\q,\q^*-\q\rangle+\frac{\rho}{2}\|\K\q\|^2.
    \end{equation*}
    Combining it with Lemma \ref{lemma:alternative to sc of f for polytopes}, and recalling \eqref{eq:derivative}, we have
    \begin{align*}
            \phi_\w(\q^*)-\phi_\w(\q) &\geq\langle\nabla f(\x),\x^*-\x\rangle+\frac{\alpha_g\sigma^{-1}}{2}\|\q^*-\q\|^2-\frac{\alpha_g}{2}\|\K\q\|^2+\frac{\rho}{2}\|\K\q^*-\K\q\|^2\\&+\langle \K^{\top}\w+\rho\K^{\top}\K\q,\q^*-\q\rangle\\
            &=\langle\nabla\phi_\w(\q),\q^*-\q\rangle+\frac{\alpha_g\sigma^{-1}}{2}\|\q^*-\q\|^2+\frac{\rho-\alpha_g}{2}\|\K\q\|^2\\
            &\geq\langle\nabla\phi_\w(\q),\q^*-\q\rangle+\frac{\alpha_g\sigma^{-1}}{2}\|\q^*-\q\|^2,
    \end{align*}
    where the last inequality is true since $\rho\geq\alpha_g$, which proves the desired result with $\as=\alpha_g\sigma^{-1}$ (recalling $\phi_\w(\q)=S(\q,\w)$).
\end{proof}

\section{Proof of Theorem \ref{thm:polytope implementaion} (weak proximal oracle for polytopes)}\label{sec:polytope oracle proof}
\begin{proof}[Proof of Theorem \ref{thm:polytope implementaion}]
We present a proof that the point $\vv_\x$ as defined in the theorem indeed satisfies \eqref{eq:x wpo}. The proof that $\vv_\y$ satisfies \eqref{eq:y wpo} follows the exact same arguments with the obvious modifications. We build on an observation from \cite{garber2016linearly}, that given a point $\x\in\F$, that is formed by a convex combination of $t$ vertices of $\F$, a vector $\p\in\R{d}$, and a radius $r\in\R{}_+$, there exists a point $\Tilde{\x}\in\F$ such that:
    \begin{enumerate}
        \item\label{eq:linear} $\langle\Tilde{\x},\p\rangle\leq\langle\z,\p\rangle$ $\forall\z\in\F\cap B(\x,r)$,
        \item\label{eq:local}$\|\x-\Tilde{\x}\|\leq\omega\cdot r$,
        \item\label{convex hull}$\tilde{\x}$ is in the convex hull of the $t$ vertices needed to represent $\x$, and the vertex which is the output of the LMO of $\F$ w.r.t. $\p$,
    \end{enumerate}
    where $B(\x,r)$ is a ball of radius $r$ centered at $\x$ and $\omega\geq1$ is some constant that depends on the geometry of $\F$ (see \cite{garber2016linearly} Section 2).

    For any $\uu\in\F$, let $\tilde{\x}_{\uu}$ be a point satisfying the above three properties for $r=\|\uu-\x\|$ and $\p=\p_{\x}:=\nabla_\x S(\q,\w)+2\mu\A^{\top}\K\q$.
    Due to Property 3 above, $\Tilde{\x}_{\uu}$ can be written as a linear combination of the vertices $\z_1,\dots,\z_{t+1}$, i.e.,  $\Tilde{\x}_{\uu}=\MM\boldsymbol\gamma$, for some $\boldsymbol\gamma\in\R{t+1}$ in the simplex (recall that the columns of $\MM$ are $\z_1,\dots,\z_{t+1}$). Therefore, by the choice  $\vv_\x=\MM\gamma^*_\x$, where $\gamma^*_\x$ is a minimizer of \eqref{pblm:x simplex}, we have for any $\uu\in\F$, that
    \begin{equation}
        \begin{split}
            \langle\vv_\x,\p\rangle+\frac{\eta(\bs+2\mu\|\K\|^2)}{2}\|\vv_\x-\x\|^2\leq\langle\Tilde{\x}_{\uu},\p\rangle+\frac{\eta(\bs+2\mu\|\K\|^2)}{2}\|\Tilde{\x}_{\uu}-\x\|^2.
        \end{split}
    \end{equation}
    By combining Properties 1 and 2 above, we have for any $\uu\in\F$, that
    \begin{equation}
        \langle\Tilde{\x}_{\uu},\p\rangle+\frac{\eta(\bs+2\mu\|\K\|^2)}{2}\|\Tilde{\x}_{\uu}-\x\|^2\leq\langle\uu,\p\rangle+\frac{\omega^2\eta(\bs+2\mu\|\K\|^2)}{2}\|\uu-\x\|^2.
    \end{equation}
    Combining the last two inequalities, we have that \eqref{eq:x wpo} is indeed satisfied with $\Rx(\cdot)=\delta_\F(\cdot)$ and $\lambda_{\x}=\omega^2$ for any $\uu\in\F$, and in particular for any $\x^*\in\X^*\subseteq\F$.
\end{proof}

\end{document}